\documentclass[12pt]{amsart} 

\usepackage[margin=1.0in]{geometry}
\usepackage{fourier} 
\usepackage{dsfont} 
\usepackage{amssymb} 
\usepackage[utf8]{inputenc} 
\usepackage[english]{babel} 

\theoremstyle{plain} 
\newtheorem{thm}{Theorem} 
\newtheorem{lem}[thm]{Lemma} 
\newtheorem{cor}[thm]{Corollary} 

\theoremstyle{definition} 
\newtheorem*{defn}{Definition}

\theoremstyle{remark} 
\newtheorem*{rem}{Remark} 

\DeclareMathOperator{\supp}{supp} 
\DeclareMathOperator{\charac}{char} 
\DeclareMathOperator{\mre}{Re} 
\DeclareMathOperator{\mim}{Im}

\def\DD{\mathbb D} 
\def\CC{\mathbb C} 
\def\RR{\mathbb R} 
\def\TT{\mathbb T} 
\def\NN{\mathbb N}
\def\veps{\varepsilon}

\title{Composition operators and embedding theorems for some function spaces of Dirichlet series} 
\date{\today} 

\author{Fr\'{e}d\'{e}ric Bayart} \address{Clermont Universit\'{e}, Universit\'{e} Blaise Pascal, Laboratoire de Math\'{e}matiques, UMR 6620 - CNRS, Campus des C\'{e}zeaux, 3, Place Vasarely, TSA 60026, CS 60026, 63178 Aubi\`{e}re cedex} \email{bayart@math.univ-bpclermont.fr}

\author{Ole Fredrik Brevig} \address{Department of Mathematical Sciences, Norwegian University of Science and Technology (NTNU), NO-7491 Trondheim, Norway} \email{ole.brevig@math.ntnu.no}

\thanks{The second author is supported by Grant 227768 of the Research Council of Norway.}

\subjclass[2010]{Primary 47B33. Secondary 30B50, 30H10, 30H20.}

\begin{document}
\begin{abstract}
	We observe that local embedding problems for certain Hardy and Bergman spaces of Dirichlet series are equivalent to boundedness of a class of composition operators.  Following this, we perform a careful study of such composition operators generated by polynomial symbols $\varphi$ on a scale of Bergman--type Hilbert spaces $\mathcal{D}_\alpha$. We investigate the optimal $\beta$ such that the composition operator $\mathcal{C}_\varphi$ maps $\mathcal{D}_\alpha$ boundedly into $\mathcal{D}_\beta$. We also prove a new embedding theorem for the non-Hilbertian Hardy space $\mathcal H^p$ into a Bergman space in the half-plane and use it to consider composition operators generated by polynomial symbols on $\mathcal{H}^p$, finding the first non-trivial results of this type. The embedding also yields a new result for the functional associated to the multiplicative Hilbert matrix.
\end{abstract}

\maketitle

\section{Introduction} \label{sec:intro} A paper by Gordon and Hedenmalm \cite{GH99} initiated the study of composition operators acting on function spaces of Dirichlet series, $f(s)=\sum_{n\geq1}a_nn^{-s}$. Their object of study was the Hilbert space of Dirichlet series with square-summable coefficients, $\mathcal{H}^2$. In this paper, we consider composition operators acting on various scales of function spaces of Dirichlet series.

For $1\leq p < \infty$, we follow \cite{Bayart02} and define the Hardy space $\mathcal{H}^p$ as the Banach space completion of Dirichlet polynomials $P(s) = \sum_{n=1}^N a_n n^{-s}$ in the Besicovitch norm 
\begin{equation}
	\label{eq:besicovitch} \|P\|_{\mathcal{H}^p} := \lim_{T\to\infty}\left(\frac{1}{2T}\int_{-T}^T |P(it)|^p\,dt \right)^\frac{1}{p}. 
\end{equation}
The spaces $\mathcal{H}^p$ are Dirichlet series analogues of the classical Hardy spaces in unit disc. We refer to \cite{Q15} and to \cite[Ch.~6]{QQ13} for basic properties of $\mathcal{H}^p$, mentioning for the moment only that their elements are absolutely convergent in the half-plane $\mathbb{C}_{1/2}$, where $\mathbb C_\theta := \{s\in \mathbb{C} \,:\, \mre(s)>\theta\}$. 

For $\alpha \in \mathbb{R}$, we let $\mathcal{D}_\alpha$ denote the Hilbert space consisting of Dirichlet series $f$ satisfying 
\begin{equation}
	\label{eq:Dalpha} \|f\|_{\mathcal{D}_\alpha} := \left(\sum_{n=1}^\infty \frac{|a_n|^2}{[d(n)]^\alpha}\right)^\frac{1}{2}<\infty. 
\end{equation}
Here $d(n)$ denotes the number of divisors of the positive integer $n$. Note that $\mathcal{D}_0=\mathcal{H}^2$. We are interested in the range $\alpha\geq0$ and, as explained in \cite{BB14}, these spaces may be thought of as Dirichlet series analogues of the classical scale of weighted Bergman spaces in the unit disc. Since $d(n)=\mathcal{O}\left(n^\varepsilon\right)$ for every $\varepsilon>0$, it follows from the Cauchy--Schwarz inequality that Dirichlet series in $\mathcal{D}_\alpha$ also are absolutely convergent in $\mathbb{C}_{1/2}$.

Due to an insight of H.~Bohr (see Section~\ref{sec:bohr}), both $\mathcal{H}^p$ and $\mathcal{D}_\alpha$ can be identified with certain function spaces in countably infinite number of complex variables, and --- consequently --- the norms \eqref{eq:besicovitch} and \eqref{eq:Dalpha} can be computed as integrals on the polytorus $\mathbb{T}^\infty$ or in the polydisc $\mathbb{D}^\infty$, respectively.

In an attempt to better understand these spaces, their composition operators $\mathcal C_{\varphi}(f)=f\circ\varphi$ have recently been investigated in a series of papers. It is well-known (see \cite{BB14,Bayart02,GH99,QS15}) that any function $\varphi \colon \mathbb{C}_{1/2} \to \mathbb{C}_{1/2}$ defining a bounded composition operator from $\mathcal{H}^p$ to $\mathcal{H}^q$, for some $p,q\geq1$, or from $\mathcal{D}_\alpha$ to $\mathcal{D}_\beta$, for some $\alpha,\beta\geq0$, necessarily is a member of the following class.
\begin{defn}
	The \emph{Gordon--Hedenmalm class}, denoted $\mathcal{G}$, is the set of functions $\varphi\colon\mathbb{C}_{1/2}\to\mathbb{C}_{1/2}$ of the form
	\[\varphi(s) = c_0 s + \sum_{n=1}^\infty c_n n^{-s} =: c_0s + \varphi_0(s),\]
	where $c_0$ is a non-negative integer called the \emph{characteristic} of $\varphi$ and is denoted $\charac(\varphi)$, the Dirichlet series $\varphi_0$ converges uniformly in $\mathbb{C}_\varepsilon$ $(\varepsilon>0)$ and has the following mapping properties: 
	\begin{itemize}
		\item[(a)] If $c_0=0$, then $\varphi_0(\mathbb{C}_0)\subset \mathbb{C}_{1/2}$. 
		\item[(b)] If $c_0\geq1$, then either $\varphi_0\equiv 0$ or $\varphi_0(\mathbb{C}_0)\subset \mathbb{C}_0$. 
	\end{itemize}
\end{defn}

Regarding sufficient conditions, the case $\charac(\varphi)\geq1$ is the best understood. It was shown in \cite{Bayart02} that (b) is sufficient for boundedness of $\mathcal{C}_\varphi$ from $\mathcal{H}^p$ to $\mathcal{H}^p$ and in \cite{BB14} that the same holds for boundedness of $\mathcal{C}_\varphi$ from $\mathcal{D}_\alpha$ to $\mathcal{D}_\alpha$.

The case $\charac(\varphi)=0$, which is the topic of this paper, is more difficult. Here it is only known that (a) is sufficient for boundedness of $\mathcal{C}_\varphi$ from $\mathcal{H}^p$ to $\mathcal{H}^p$ if $p$ is an even integer. In \cite{BB14}, it was shown that if $\varphi\in\mathcal{G}$ with $\charac(\varphi)=0$, then $\mathcal C_\varphi$ maps $\mathcal D_\alpha$ into $\mathcal D_{2^\alpha-1}$ (which is smaller than $\mathcal D_\alpha$ if $0<\alpha<1$ and larger than $\mathcal D_\alpha$ if $\alpha>1$). It was left open whether the value $2^\alpha-1$ is optimal or not.

The sticking point seems to be that in order to prove sufficient conditions for boundedness of composition operators with $\charac(\varphi)=0$, we require an embedding of the function spaces of Dirichlet series into certain classical function spaces in the half-plane $\mathbb{C}_{1/2}$. The existence of embeddings in the non-Hilbertian case is a well-known open problem in the field.

This paper is initiated by the observation that such embeddings are in fact equivalent to the sufficiency of condition (a). The precise statement of this equivalence is presented in Theorem~\ref{thm:Hp} (for $\mathcal{H}^p$) and Theorem~\ref{thm:Dequiv} (for $\mathcal{D}_\alpha$) below. Our approach is related to the transference principle introduced in \cite{QS15}. As a corollary, we obtain that the parameter $2^\alpha-1$ discussed above is sharp, since it was demonstrated in \cite{Olsen11} that the corresponding embedding is optimal. 

We also discuss embeddings of $\mathcal H^p$ when $1\leq p<2$. Although we were unable to prove that $\mathcal H^p$ embeds into the corresponding conformally invariant Hardy space of $\CC_{1/2}$, we show that it embeds into an optimal conformally invariant Bergman space.

\begin{thm}\label{thm:bergmanemb}
Let $1\leq p<2$. There exists a constant $C_p>0$ such that
\[\left(\int_{\mathbb{R}}\int_{1/2}^\infty |f(s)|^2\, \left(\sigma-\frac{1}{2}\right)^{\frac{2}{p}-2}\frac{d\sigma dt}{|s+1/2|^{4/p}}\right)^\frac{1}{2}\leq C_p \|f\|_{\mathcal{H}^p},\]
for every $f\in\mathcal{H}^p$. The exponent $\frac{2}{p}-2$ is the smallest possible.
\end{thm}

We then perform a careful study of composition operators with polynomial symbols mapping $\mathcal{D}_\alpha$ to $\mathcal{D}_\beta$, in the spirit of \cite{BB15}. We show that for certain polynomial symbols, $\mathcal C_\varphi$ maps $\mathcal D_\alpha$ into $\mathcal D_{\beta}$ with $\beta<2^\alpha-1$ and that the optimality of $\beta=2^\alpha-1$ also can be decided by investigating the most simple non-trivial symbol, namely $\varphi(s)=3/2-2^{-s}$. 

Consequently, we consider boundedness of this simple composition operator an interesting necessary condition for the embedding problem for $\mathcal{H}^p$. This leads us to an in-depth study of composition operators with linear symbols on $\mathcal H^p$. By using Theorem~\ref{thm:bergmanemb} and estimates of Carleson measures, we prove the following result.

\begin{thm}\label{thm:linearHp}
Let $\varphi(s) = c_1 + \sum_{j=1}^d c_{p_j}p_j^{-s}$ be a Dirichlet polynomial supported on the primes such that $c_{p_j}\neq 0$ for $j=1,\,\dots,\,d$. If $\varphi \in \mathcal{G}$ and $d\geq2$, then $\mathcal C_\varphi$ is bounded on $\mathcal H^p$ for every $p \in [1,\infty)$.
\end{thm}

Observe that the case $d=1$ corresponds to the simple symbol discussed above. It should also be mentioned that very few non-trivial composition operators of characteristic $0$ on $\mathcal H^p$ are known when $p$ is not an even integer, and none involving two or more prime numbers. Moreover, it is possible to generate more examples from our method and results in \cite{BB15}.

We finally show that if $\varphi(s)=3/2-2^{-s}$ generates a bounded composition operator on $\mathcal{H}^1$, then Nehari's theorem holds for the multiplicative Hilbert matrix introduced in \cite{BPSSV14}. We apply Theorem~\ref{thm:bergmanemb} to demonstrate that the associated functional is bounded on $\mathcal{H}^p$ for $p\in(1,\infty)$.

\subsection*{Organization} This paper is divided into six sections. \begin{itemize}
	\item Section~\ref{sec:bohr} contains an exposition of our observation that the local embedding problem mentioned above is equivalent to boundedness of certain composition operators for $\mathcal{H}^p$ (Theorem~\ref{thm:Hp}) and $\mathcal{D}_\alpha$ (Theorem~\ref{thm:Dequiv}), in addition to the proof of Theorem \ref{thm:bergmanemb}. 
	\item In Section~\ref{sec:carleson}, we collect some results regarding Carleson measures in the half-plane and on the polydisc, which will be needed in the following sections.
	\item Section~\ref{sec:poly} is devoted to a study of composition operators from $\mathcal{D}_\alpha$ to $\mathcal{D}_\beta$ generated by polynomial symbols. The main result of this section, Theorem~\ref{thm:dirichletpolynomial}, demonstrates that the boundedness of $\mathcal{C}_\varphi \colon \mathcal{D}_\alpha \to \mathcal{D}_\beta$ depends strongly on the complex dimension and degree of the polynomial symbol.
	\item In Section~\ref{sec:hp}, we discuss composition operators with linear symbols on $\mathcal{H}^p$. The proof of Theorem~\ref{thm:linearHp} can be found here.
	\item The final section contains some connections from the results obtained in this paper to the validity of Nehari's theorem for the multiplicative Hilbert matrix.
\end{itemize}

\subsection*{Notation} We will use the notation $f(x)\ll g(x)$ when there is some constant $C>0$ such that $|f(x)|\leq C|g(x)|$ for all (appropriate) $x$. If both $f(x)\ll g(x)$ and $g(x)\ll f(x)$ hold, we will write $f(x)\asymp g(x)$. As usual, $\{p_j\}_{j\geq1}$ will denote the increasing sequence of prime numbers.

\section{Composition operators and the embedding problem} \label{sec:bohr} 
\subsection{Hardy spaces}
As mentioned in the introduction, functions in $\mathcal{H}^p$ are holomorphic in the half-plane $\mathbb{C}_{1/2}$. It is therefore interesting to investigate how they behave on the line $1/2+it$. In this context, the most important question is the embedding problem (see \cite[Sec.~3]{SS09}), which can be formulated as follows. Is there a constant $C_p$ such that 
\begin{equation}
	\label{eq:localemb} \sup_{\tau \in \mathbb{R}}\int_{\tau}^{\tau+1} |P(1/2+it)|^p\,dt \leq C_p \|P\|_{\mathcal{H}^p}^p 
\end{equation}
for every Dirichlet polynomial $P$? It follows from an inequality of Montgomery and Vaughan (see \cite[pp.~140--141]{Montgomery94}) that \eqref{eq:localemb} holds for $p=2$, and hence for every even integer $p$, but its validity for other values remains open. Now, from \eqref{eq:besicovitch} it is clear that the $\mathcal{H}^p$ norm is invariant under vertical translations, so it is enough to check \eqref{eq:localemb} for a fixed $\tau$, say $\tau=0$.

A typical (see e.g.~\cite{Bayart02,GH99}) application of the local embedding is to deduce that if $\varphi$ is in $\mathcal{G}$ with $\charac(\varphi)=0$, then the composition operator $\mathcal{C}_\varphi$ is bounded on $\mathcal{H}^p$. This is usually done through the following equivalent formulation of \eqref{eq:localemb}. 

The conformally invariant Hardy space in the half-plane $\mathbb{C}_{1/2}$, which we denote $H^p_{\operatorname{i}}$, consists of those functions $f$ such that $f \circ \mathcal{T} \in H^p(\mathbb{T})$, where $\mathcal{T}$ is the following mapping from $\mathbb{D}$ to $\mathbb{C}_{1/2}$,
\[\mathcal{T}(z) = \frac{1}{2}+\frac{1-z}{1+z}.\]
The mapping $\mathcal{T}$ appeared in the transference principle of \cite{QS15}, where it was used to transfer certain results about composition operators on $H^2(\mathbb{T})$ to results about composition operators on $\mathcal{H}^2$. Now, the norm of $H^p_{\operatorname{i}}$ can be computed as 
\begin{equation}
	\label{eq:Hpi} \|f\|_{H^p_{\operatorname{i}}}^p := \|f\circ \mathcal{T}\|_{H^p(\mathbb{T})}^p = \frac{1}{2\pi} \int_{-\pi}^{\pi} |f(1/2+i\tan(\theta/2))|^p \, d\theta = \frac{1}{\pi}\int_{\mathbb{R}} |f(1/2+it)|^p \,\frac{dt}{1+t^2}. 
\end{equation}
The inequality \eqref{eq:localemb} is equivalent to $\|P\|_{H^p_{\operatorname{i}}} \leq C_p^\prime\|P\|_{\mathcal{H}^p}$, since evidently
\[\int_0^1 |P(1/2+it)|^p\, dt \ll \|P\|_{H^p_{\operatorname{i}}}^p \ll \sup_{\tau \in \mathbb{R}} \int_{\tau}^{\tau+1} |P(1/2+it)|^p\,dt.\]

Our observation is that not only does the embedding \eqref{eq:localemb} imply a sufficient condition for boundedness of certain composition operators, it is in fact equivalent to boundedness of all composition operators of this type.
\begin{thm}
	\label{thm:Hp} Fix $1\leq p < \infty$. The following are equivalent. 
	\begin{itemize}
		\item[(a)] The local embedding \eqref{eq:localemb} holds for $p$. 
		\item[(b)] For every $\varphi \in \mathcal{G}$ with $\charac(\varphi)=0$, the composition operator $\mathcal{C}_\varphi$ acts boundedly on $\mathcal{H}^p$. 
		\item[(c)] Let $\psi(s)=\mathcal{T}(2^{-s})$. The composition operator $\mathcal{C}_\psi$ acts boundedly on $\mathcal{H}^p$. 
	\end{itemize}
\end{thm}
As explained in \cite{Bayart02}, the proof of (a) $\implies$ (b) can be adapted from the proof given for $p=2$ in \cite{GH99}. This argument relies on approximating the Besicovitch norm \eqref{eq:besicovitch} by taking a limit in a family of conformal mappings. A simpler proof of this implication, based on a trick from \cite{BB14}, is included below. 

To facilitate this, let us recall the Bohr lift. Every positive integer $n$ can be written uniquely as a product of prime numbers,
\[n = \prod_{j=1}^\infty p_j^{\kappa_j}.\]
This factorization associates the finite multi-index $\kappa(n)=(\kappa_1,\,\kappa_2,\,\ldots\,)$ to $n$. Consider a Dirichlet series $f(s)=\sum_{n\geq 1}a_n n^{-s}$. Its Bohr lift $\mathcal{B}f$ is the power series
\[\mathcal{B}f(z) = \sum_{n=1}^\infty a_n z^{\kappa(n)}.\]
It is well-known (see \cite{Bayart02,QQ13}) that the Bohr lift defines an isometric isomorphism between $\mathcal{H}^p$ and the Hardy space of the countably infinite polytorus, $H^p(\mathbb{T}^\infty).$ The polytorus $\mathbb{T}^\infty$ is a compact abelian group, which we endow with its normalized Haar measure $\nu$, so that 
\[\|f\|_{\mathcal{H}^p}^p = \|\mathcal{B}f\|_{H^p(\mathbb{T}^\infty)}^p := \left(\int_{\mathbb{T}^\infty} |\mathcal{B}f(z)|^p\,d\nu(z)\right)^\frac{1}{p}. \]
It is important to note that the Haar measure $\nu=\nu_0$ of the polytorus $\mathbb{T}^\infty$ is simply the product of the normalized Lebesgue measure on $\mathbb{T}$, denoted $m=m_0$, in each variable. The subscript is included to indicate the connection to $\mathcal{D}_0=\mathcal{H}^2$.
\begin{proof}
	[Proof of Theorem~\ref{thm:Hp}] For (a) $\implies$ (b), we first suppose that $\Phi$ is a holomorphic function mapping $\mathbb{D}$ to $\mathbb{C}_{1/2}$. Using Littlewood's subordination principle (see \cite[Ch.~11]{Zhu}), we find that 
	\begin{equation}
		\label{eq:LWS} \|f\circ \Phi\|_{H^p(\mathbb{T})}^p \leq \frac{1+|\mathcal{T}^{-1}(\Phi(0))|}{1-|\mathcal{T}^{-1} (\Phi(0))|}\|f\|_{H^p_{\operatorname{i}}}^p,
	\end{equation}
for $f\in H^p_{\operatorname{i}}$. For $G \in H^p(\mathbb{T}^\infty)$ and $w \in \mathbb{C}$, set $G_w(z) = G(wz_1,\,wz_2,\,\ldots\,)$. By Fubini's theorem, 
	\[\|G\|_{H^p(\mathbb{T}^\infty)}^p = \int_{\mathbb{T}^\infty} \int_{\mathbb{T}} |G_w(z)|^p\,d m(w) d\nu(z).\]
	Let $P$ be a Dirichlet polynomial and assume that $\varphi \in \mathcal{G}$ with $\charac(\varphi)=0$. The latter assumption implies that $\mathcal{B}(P\circ \varphi) = P \circ (\mathcal{B}\varphi)$. Thus, by setting $G = \mathcal{B}(P\circ \varphi)$, we obtain
	\[\|P \circ \varphi\|_{\mathcal{H}^p}^p = \int_{\mathbb{T}^\infty} \int_{\mathbb{T}} |P \circ (\mathcal{B}\varphi)_w(z)|^p\, dm(w) d\nu(z).\]
	Fixing for a moment $z \in \mathbb{T}^\infty$, we notice that $\Phi(w) = (\mathcal{B}\varphi)_w(z)$ maps $\mathbb{D}$ to $\mathbb{C}_{1/2}$ with $\Phi(0)= c_1$. Considering therefore $P$ a member of $H^p_{\operatorname{i}}$, we apply \eqref{eq:LWS} and conclude that
	\[\|P\circ\varphi\|_{\mathcal{H}^p}^p \leq \int_{\mathbb{T}^\infty} \left(\frac{1+|\mathcal{T}^{-1}(c_1)|}{1-|\mathcal{T}^{-1}(c_1)|}\|P\|_{H^p_{\operatorname{i}}}^p\right)\, d\nu(z) = \frac{1+|\mathcal{T}^{-1}(c_1)|}{1-|\mathcal{T}^{-1}(c_1)|}\|P\|_{H^p_{\operatorname{i}}}^p,\]
	seeing as the constant in this instantiation of Littlewood's subordination principle does not involve $z$.
	
	The implication (b) $\implies$ (c) is obvious, seeing as it is easy to verify that $\psi \in \mathcal{G}$. To prove that (c) $\implies$ (a), assume that $\mathcal{C}_\psi$ acts boundedly on $\mathcal{H}^p$, say that
	\[\|\mathcal{C}_\psi P\|_{\mathcal{H}^p} \leq C_p \|P\|_{\mathcal{H}^p}\]
	holds for every Dirichlet polynomial $P$. Arguing as above, we find that $\mathcal{B}(P\circ \psi) = P \circ (\mathcal{B}\psi)$
	and that, in this case, $\mathcal{B}\psi(z)=\mathcal{T}(z_1)$. In particular, using the Bohr lift, this means that 
	\[\|\mathcal{C}_\psi P\|_{\mathcal{H}^p}  = \|P \circ \mathcal{T}\|_{H^p(\mathbb{T})},\]
	so we are done by \eqref{eq:Hpi}.
\end{proof}

\subsection{Bergman spaces} Let us now explain how to do the same for the Bergman--type spaces $\mathcal{D}_\alpha$. Let $\alpha,\beta>0$, and consider the following probability measures on $\mathbb{D}$. 
\begin{align}
	dm_\alpha(z) &= \frac{1}{\Gamma(\alpha)}\left(\log{\frac{1}{|z|^2}}\right)^{\alpha-1} \,dm_1(z), \label{eq:malpha}\\
	d\widetilde{m}_\beta(z) &= \beta\left(1-|z|^2\right)^{\beta-1}\,dm_1(z). \label{eq:mbeta} 
\end{align}
Here $m_1$ (which is the only case where $m=\widetilde{m}$) is taken to be the standard Lebesgue measure on $\mathbb{C}$, normalized so that $m_1(\mathbb{D})=1$. For $\alpha>0$, the Bergman space $D_\alpha(\mathbb{D})$ can be defined as the $L^2$-closure of polynomials with respect to either measure, yielding equivalent norms. We will for simplicity use the measure \eqref{eq:mbeta} in most cases.

However, in an infinite number of variables, the norms are no longer equivalent. We use \eqref{eq:malpha} to compute the norm of $\mathcal{D}_\alpha$ as an integral over $\mathbb{D}^\infty$ to ensure that \eqref{eq:Dalpha} is satisfied. Therefore, we define $d\nu_\alpha(z) = dm_\alpha(z_1)\times dm_\alpha(z_2)\times \cdots$. It is straightforward to verify that
\[\|f\|_{\mathcal{D}_\alpha}^2 = \int_{\mathbb{D}^\infty} |\mathcal{B}f(z)|^2\,d\nu_\alpha(z).\]
Set $S_\tau = [1/2,1]\times[\tau,\tau+1]$. For the Bergman spaces $\mathcal{D}_\alpha$, the local embedding problem takes on the following form: Given $\alpha>0$, what is the smallest $\beta>0$ such that 
\begin{equation}
	\label{eq:blocal} \sup_{\tau \in \mathbb{R}}\int_{S_\tau} |P(s)|^2\,\left(\sigma-\frac{1}{2}\right)^{\beta-1}\, dm_1(s) \leq C_{\alpha,\beta} \|P\|_{\mathcal{D}_\alpha}^2
\end{equation}
for every Dirichlet polynomial $P$? Again, it is clear that the norm of $\mathcal{D}_\alpha$ is invariant under vertical translations, so arguing as above, we find that \eqref{eq:blocal} is equivalent to $\|P\|_{D_{\beta,\operatorname{i}}} \leq C^\prime_{\alpha,\beta} \|P\|_{\mathcal{D}_\alpha}^2$, setting 
\begin{equation}
	\label{eq:DBi} \|f\|_{D_{\beta,\operatorname{i}}}^2 := \|f\circ\mathcal{T}\|_{D_\beta(\mathbb{D})}^2 = 4^\beta \beta \int_{\mathbb{C}_{1/2}} |f(s)|^2 \,\left(\sigma-\frac{1}{2}\right)^{\beta-1}\, \frac{dm_1(s)}{|s+1/2|^{2\beta+2}}, 
\end{equation}
since any $f$ in $\mathcal D_\alpha$ is uniformly bounded in $\mathbb{C}_{1}$ by its $\mathcal D_\alpha$ norm. For the next result, (a) $\implies$ (b) is part of the main result in \cite{BB14}. The other steps are identical to the proof of Theorem~\ref{thm:Hp} in view of the discussion above.
\begin{thm}
	\label{thm:Dequiv} Fix $\alpha,\beta>0$. The following are equivalent. 
	\begin{itemize}
		\item[(a)] The local embedding \eqref{eq:blocal} holds for $\alpha$ and $\beta$. 
		\item[(b)] For every $\varphi \in \mathcal{G}$ with $\charac(\varphi)=0$, the composition operator $\mathcal{C}_\varphi \colon \mathcal{D}_\alpha \to \mathcal{D}_\beta$ is bounded. 
		\item[(c)] Let $\psi(s)=\mathcal{T}(2^{-s})$. The composition operator $\mathcal{C}_\psi$ maps $\mathcal{D}_\alpha$ boundedly into $\mathcal{D}_\beta$. 
	\end{itemize}
\end{thm}
It was shown in \cite{Olsen11} that $\beta = 2^\alpha-1$ is the optimal exponent in \eqref{eq:blocal}. We will touch upon the reason behind this value in the next section, see in particular \eqref{eq:wilsonfactor}. From this optimality, we obtain at once the following result, clarifying the optimal $\beta$ in the main result of \cite{BB14}, which states that if $\varphi\in\mathcal{G}$ with $\charac(\varphi)=0$, then $\mathcal{C}_\varphi$ maps $\mathcal{D}_\alpha$ boundedly into $\mathcal{D}_\beta$ if $\beta\geq 2^\alpha-1$.
\begin{cor}
	\label{thm:optimality} Let $\alpha\geq 0$. There is $\varphi \in \mathcal{G}$ with $\charac(\varphi)=0$ such that $\mathcal{C}_\varphi \colon \mathcal{D}_\alpha \to \mathcal{D}_\beta$ is bounded if and only if $\beta \geq 2^{\alpha}-1$. 
\end{cor}

\subsection{Embedding of $\mathcal H^p$ into $D_{\beta,\operatorname{i}}$}
Even if one is unable to prove the embedding inequality \eqref{eq:localemb} for $1\leq p<2$, it is natural to ask whether it is possible to embed $\mathcal H^p$ into some Bergman space $D_{\beta,\operatorname{i}}$. For the Hardy spaces of the unit disc, this type of result goes back to the function theoretic version of the isoperimetric inequality due to Carleman, which asserts that
\begin{equation} \label{eq:carleman}
	\|f\|_{D_1(\DD)}\leq \|f\|_{H^1(\DD)}.
\end{equation}
Iterating the inequality (its contractivity is crucial) and using the Bohr lift, Helson \cite{Helson06} found that $\|f\|_{\mathcal D_1}\leq \|f\|_{\mathcal H^1}.$
Combining Helson's inequality with the results from \cite{Olsen11} discussed above, one finds that $\mathcal{H}^1$ is embedded in $D_{1,\operatorname{i}}$, thereby reclaiming \eqref{eq:carleman} in the context of Hardy spaces of Dirichlet series and weighted Bergman spaces in $\mathbb{C}_{1/2}$.

If we seek to extend Helson's inequality to $1<p<2$, we are required to use the measure \eqref{eq:malpha} when defining the spaces $D_\alpha(\mathbb{D})$, to ensure that  we get $\mathcal{D}_\alpha$ after the iterative procedure. By a standard interpolation argument between \eqref{eq:carleman} and $H^2(\mathbb{D})$, one find that for $p\in(1,2)$, 
\begin{equation}\label{eq:interpolation}
\|f\|_{D_{\frac 2p-1}(\DD)}\leq C_p \|f\|_{H^p(\DD)}.
\end{equation}
Nevertheless, the constant $C_p$ arising from interpolation between Hardy spaces is strictly bigger than $1$ (see \cite{BHS15}). Without contractivity, we cannot argue as Helson, starting from \eqref{eq:interpolation}, to prove that $\mathcal H^p$ embeds into $\mathcal D_{2/p-1}$. It turns out that this embedding is false, since it can be proved (see \cite{BHS15} or the argument at the end of the proof of Theorem \ref{thm:bergmanemb}) that if $\mathcal H^p$ embeds into $\mathcal D_\alpha$, then $\alpha\geq 1-\log p/\log 2$ which is stricly bigger than $2/p-1$ when $p\in(1,2)$.

On the other hand, such an embedding is not known to exist, unless $p\in\{1,2\}$. If we could prove that $\mathcal H^p$ embeds into $\mathcal{D}_\alpha$, with $\alpha=1-\log p/\log 2$, then the embedding \eqref{eq:blocal}, which is valid with $\beta=2^\alpha-1$, would imply that
\begin{equation} \label{eq:Hbin}
	\|f\|_{D_{\frac{2}{p}-1,\operatorname{i}}} \ll \|f\|_{\mathcal{H}^p},
\end{equation}
again reclaiming \eqref{eq:interpolation} for Hardy spaces of Dirichlet series and weighted Bergman spaces in $\mathbb{C}_{1/2}$. Similarly, the embedding \eqref{eq:localemb} also implies \eqref{eq:Hbin}, in this case by first translating \eqref{eq:interpolation} to $\mathbb{C}_{1/2}$ with $\mathcal{T}$. We have been able to prove \eqref{eq:Hbin} by different methods, which is our Theorem~\ref{thm:bergmanemb}.

The proof uses several tools from harmonic analysis and analytic number theory. The first is a special case of a result of Weissler \cite{Weissler80}, who studied the hypercontractivity of the Poisson kernel.
\begin{lem}\label{lem:weissler}
Let $p\in[1,2]$. For any $f(z)=\sum_{k\geq 0}a_kz^k$, we have the contractive estimate
\[\left(\sum_{k=0}^\infty |a_k|^2 \left(\frac{p}{2}\right)^k\right)^{1/2}\leq \|f\|_{H^p(\mathbb{D})}.\]
\end{lem}
The second tool is a way to iterate this inequality multiplicatively, first devised in \cite{Bayart02} and later used in \cite{BHS15,Helson06}. We formulate it in an abstract context and we give a brief account of the proof.
\begin{lem}\label{lem:iteration}
Let $p\in[1,2]$ and assume that there exists a sequence $\{\gamma_k\}_{k\geq0}$ of positive real numbers with $\gamma_0=1$, such that for every $f(z)=\sum_{k\geq 0}a_kz^k\in H^p(\DD)$, 
$$\left(\sum_{k=0}^\infty|a_k|^2\,\gamma_k\right)^{1/2}\leq \|f\|_{H^p(\DD)}.$$
Let $\Gamma(n)$ denote the multiplicative function defined on the prime powers by $\Gamma(p_j^k) = \gamma_k$. Then, 
$$\left(\sum_{n=1}^\infty |a_n|^2 \,\Gamma(n)\right)^{1/2}\leq \|f\|_{\mathcal H^p},$$
for every $f(s)=\sum_{n\geq 1}a_nn^{-s}\in\mathcal H^p$.
\end{lem}

\begin{proof}
Fix $d\geq1$ and $f(z)=\sum_{\kappa \in \mathbb{N}^d} a_\kappa z^\kappa\in H^p(\TT^d)$. By the Bohr lift, it is sufficient to prove that
\begin{equation} \label{eq:Hpdin}
	\left(\sum_{\kappa\in\NN^d}|a_\kappa|^2\gamma_{\kappa_1}\cdots \gamma_{\kappa_d}\right)^{1/2}\leq \|f\|_{H^p(\TT^d)}.
\end{equation}
The assumption of the lemma is that \eqref{eq:Hpdin} holds for $d=1$. We will argue by induction on $d$ and assume that \eqref{eq:Hpdin} is true for $d-1$. Then, fixing $z_1,\dots,z_{d-1}\in\TT^{d-1}$ and considering $f$ a function only of $z_d$, we use \eqref{eq:Hpdin} with $d=1$ to get
$$\left(\int_{\TT}\left|\sum_{\kappa\in\NN^d}a_\kappa \gamma_{\kappa_d}^{1/2}z_1^{\kappa_1}\cdots z_d^{\kappa_d}\right|^2dm(z_d)\right)^{p/2}\leq \int_{\TT}\left|\sum_{\kappa\in\NN^d}a_\kappa z_1^{\kappa_1}\cdots z_d^{\kappa_d}\right|^pdm(z_d).$$
We integrate over the remaining coordinates $z_1,\dots,z_{d-1}$ and use Minkowski inequality in the following form: For measure spaces $X$ and $Y$, a measurable function $g$ on $X\times Y$ and $r\geq 1$,
$$\left(\int_X\left(\int_Y |g(x,y)|dy\right)^rdx\right)^{1/r}\leq \int_Y\left(\int_X|g(x,y)|^rdx\right)^{1/r}dy.$$
This yields, with $X=\mathbb{T}$, $Y=\mathbb{T}^{d-1}$ and $r=2/p$, that
$$\left(\int_{\TT}\left(\int_{\TT^{d-1}}\left|\sum_{\kappa\in\NN^d}a_\kappa \gamma_{\kappa_d}^{1/2}z_1^{\kappa_1}\cdots z_d^{\kappa_d}\right|^pdm(z_1)\cdots dm(z_{d-1})\right)^{2/p}dm(z_d)\right)^{p/2}\leq \|f\|^p_{H^p(\TT^d)}.$$
The induction hypothesis allows us to conclude.
\end{proof}

Our final tool is a number theoretic estimate on the average order of a multiplicative function. Let $\Omega(n)$ be the total number of prime divisors of $n$, say $\Omega(p_1^{\kappa_1}\cdots p_d^{\kappa_d})=\kappa_1+\cdots+\kappa_d$. For $0<y<2$ we refer to  Selberg--Delange method (see \cite[Thm.~II.6.2]{Tenenbaum15}) and for $y=2$ we refer to \cite{Bateman57}.
\begin{lem} \label{lem:turan}
	Let $0<y\leq2$. Then
	\begin{equation} \label{eq:avordest}
		\frac{1}{x}\sum_{n\leq x} y^{\Omega(n)} \asymp \begin{cases}
			(\log{x})^{y-1} & \text{if }\, 0<y<2, \\
			(\log{x})^2 & \text{if }\, y=2.
		\end{cases}
	\end{equation}
\end{lem}
Observe the phase change at $y=2$, which occurs since $2$ is the first prime number. We are now ready to proceed with the proof of \eqref{eq:Hbin}.
\begin{proof}[Proof of Theorem \ref{thm:bergmanemb}]
Combining Lemma~\ref{lem:weissler} and Lemma~\ref{lem:iteration}, we get the inequality
\begin{equation} \label{eq:newemb}
	\left(\sum_{n=1}^\infty |a_n|^2\,\left(\frac p2\right)^{\Omega(n)}\right)^{1/2}\leq \|f\|_{\mathcal H^p},
\end{equation}
for every $f(s)=\sum_{n\geq 1}a_nn^{-s}\in\mathcal H^p$, since in this case $\Gamma(n) = (p/2)^{\Omega(n)}$. In other words, following the conventions of \cite{Olsen11}, the space $\mathcal H^p$ is continuously embedded into $$\mathcal H_{w_p}:=\left\{\sum_{n=1}^\infty  a_nn^{-s}\,:\, \|f\|_{w_p} := \left(\sum_{n=1}^\infty |a_n|^2/w_p(n)\right)^\frac{1}{2}<\infty\right\}, \qquad \text{where} \quad w_p(n) = \left(\frac 2p\right)^{\Omega(n)}.$$
The main result of \cite{Olsen11} relates the average order of the weight $w(n)$ with the optimal embedding of $\mathcal{H}_w$ into $D_{\beta,\operatorname{i}}$, the relation being the two-sided estimate
\begin{equation} \label{eq:avorder}
	\frac{1}{x}\sum_{n\leq x} w(n) \asymp (\log{x})^{\beta}.
\end{equation}
Now, the case $p=1$ was discussed and resolved above, using Helson's inequality. For $1<p<2$, we have $1< y <2$, so we conclude using \eqref{eq:avordest} that $\mathcal{H}_{w_p}$ is continuously embedded in $\mathcal{D}_{2/p-1,\operatorname{i}}$ and that the parameter $2/p-1$ is optimal, with respect to $\mathcal{H}_{w_p}$. This proves \eqref{eq:Hbin}, using \eqref{eq:newemb}.

It remains only to verify that the optimality of the parameter $2/p-1$ extends to $\mathcal{H}^p$. Fix $\varepsilon>0$ and consider
\[f_\varepsilon(s) = \frac{[\zeta(s+1/2+\varepsilon)]^{2/p}}{[\zeta(1+2\varepsilon)]^{1/p}},\]
which, as shown in \cite[Thm.~3]{Bayart02}, satisfies $\|f_\varepsilon\|_{\mathcal{H}^p}=1$. For $s=\sigma+it$ satisfying, say, $1< \sigma <3/2$ and $0< t< 1$, we have that $\zeta(s) \asymp (s-1)^{-1}$. Assume now that $\mathcal{H}^p$ embed continuously into $D_{\beta,\operatorname{i}}$. Then, for $1\leq p<2$ and $0<\beta\leq 1$, we estimate
\begin{align*}
	1 \gg \|f_\varepsilon\|_{D_{\beta,\operatorname{i}}} &\gg \int_{1/2}^1 \int_{0}^1 \frac{|\zeta(s+1/2+\varepsilon|]^{4/p}}{[\zeta(1+2\varepsilon)]^{2/p}}\,\left(\sigma-\frac{1}{2}\right)^{\beta-1}\,dtd\sigma \\
	&\gg \varepsilon^{2/p} \int_{1/2}^1 \int_{0}^1 \frac{(\sigma-1/2)^{\beta-1}}{\left((\sigma-1/2+\varepsilon)^2+t^2\right)^{2/p}}\,dtd\sigma \\
	&\asymp \varepsilon^{2/p} \int_{1/2}^1 \frac{(\sigma-1/2)^{\beta-1}}{(\sigma-1/2+\varepsilon)^{4/p-1}}\,d\sigma \gg \varepsilon^{2/p+\beta-4/p+1},
\end{align*}
which means that if $\mathcal{H}^p$ is continuously embedded in $D_{\beta,\operatorname{i}}$, then necessarily $\beta\geq 2/p-1$.
\end{proof}

Let us compare the space $\mathcal H_{w_p}$ to the space $\mathcal D_\alpha$ for $\alpha=1-\log p/\log 2$. It turns out that if $n$ is square-free, then $(p/2)^{\Omega(n)} = 1/[d(n)]^{\alpha}$. For other values, $w_p(n)$ is strictly smaller than $1/[d(n)]^{\alpha}$, and it can be significantly smaller, most easily seen by considering $n=2^k$. Thus, the space $\mathcal{H}_{w_p}$ is (strictly) bigger than $\mathcal D_\alpha$. However, when $1<p<2$, the weights $w_p(n)$ are dominated by their square-free parts, so $\mathcal{D}_\alpha$ and $\mathcal{H}_{w_p}$ are embedded into the same $\mathcal{D}_\beta$.

To explain why this happens, let $\xi$ be any positive multiplicative function with $\xi(p_j)=\beta$ and $\xi(p_j^k) \ll (2-\delta)^k$ for some $0<\delta<2$. Then, for $\mre(s)>1$, 
\begin{align*}
	\sum_{n=1}^\infty \xi(n) n^{-s} &= \prod_{j=1}^\infty \left(1 + \beta p_j^{-s} + \sum_{k=2}^\infty \xi(p_j^k)p_j^{-ks}\right) \\ &= [\zeta(s)]^{\beta}\prod_{j=1}^\infty \left(1+ \beta p_j^{-s} + \mathcal{O}(p_j^{-2s})\right)\left(1-\beta p_j^{-s} + \mathcal{O}(p_j^{-2s})\right) = [\zeta(s)]^{\beta}\prod_{j=1}^\infty \left(1 + \mathcal{O}(p_j^{-2s})\right),
\end{align*}
so by the Selberg--Delange method, we find $\sum_{n\leq x} \xi(n) \asymp x (\log{x})^{\beta-1}$. Observe again the phase change at $\delta=0$, leading to different embeddings for $\mathcal{H}_{w_1}$ and $\mathcal{D}_1$ in view of \eqref{eq:avorder}, since the latter weight satisfies the assumption $\xi(p_j^k)\ll (2-\delta)^k$, while the former does not.

\begin{rem}
	By using Weissler's inequality \cite{Weissler80} for $p\geq2$ and arguing as in the proof of Lemma~\ref{lem:iteration}, we find that if $f(s)=\sum_{n\geq1} a_n n^{-s}$ and $2\leq p < \infty$, then
	\begin{equation} \label{eq:weissler2}
		\|f\|_{\mathcal{H}^p} \leq \left(\sum_{n=1}^\infty |a_n|^2 \left(\frac{p}{2}\right)^{\Omega(n)}\right)^\frac{1}{2}.
	\end{equation}
	This inequality allows us to improve a result on the bounded zero sequences of $\mathcal{H}^p$ from \cite{Seip13}. We achieve this by replacing \cite[Lem.~9]{Seip13} with \eqref{eq:weissler2} and \cite[Lem.~6]{Seip13} with Lemma~\ref{lem:turan}. No additional changes to the arguments are required. In the notation of \cite{Seip13} we get that
	\begin{equation} \label{eq:seipimproved}
		Z\left(D_{1-2/p}(\mathbb{C}_{1/2})\right) \subset Z(\mathcal{H}^p),
	\end{equation}
	for $2<p<\infty$. This improves a similar statement from \cite[Sec.~4]{Seip13} when $p$ is not an even integer. 
	
	Taking the dual of \eqref{eq:interpolation} in the $H^2(\mathbb{D})$ pairing and recalling that $(H^p)^\ast \cong H^{p/(p-1)}$ for $1<p<\infty$, we find that if $f(z) = \sum_{k\geq0} a_k z^k$ and $2\leq p < \infty$, then
	\[\|f\|_{H^p(\mathbb{D})} \leq C_p \left(\sum_{k=0}^\infty |a_k|^2 (k+1)^{1-2/p}\right)^\frac{1}{2}.\]
	As in \eqref{eq:interpolation} the parameter $1-2/p$ is optimal. This indicates that \eqref{eq:seipimproved} is the best possible result of this type we can hope to obtain by Hilbert space techniques.
\end{rem}

\section{Carleson measures in the half-plane and on the polydisc} \label{sec:carleson} 
\subsection{Carleson measures in the half-plane} The non-conformal Bergman space $D_\beta(\mathbb{C}_{1/2})$, for $\beta>0$, consists of the holomorphic functions $f$ in $\mathbb{C}_{1/2}$ which satisfy
\[\|f\|_{D_\beta(\mathbb{C}_{1/2})}^2 := \int_{\mathbb{C}_{1/2}} |f(s)|^2\left(\sigma-\frac{1}{2}\right)^{\beta-1}\,ds < \infty.\]
If $\beta=0$, then $D_\beta(\mathbb{C}_{1/2})$ is taken to be the non-conformal Hardy space, $H^2(\mathbb{C}_{1/2})$, with norm
\[\|f\|_{H^2(\mathbb{C}_{1/2})}^2 := \sup_{\sigma>1/2} \int_{\mathbb{R}} |f(\sigma+it)|^2\, dt<\infty.\]
For $\alpha,\beta\geq0$, let $X$ denote either $\mathcal{D}_\alpha$ or $D_\beta(\mathbb{C}_{1/2})$. A positive Borel measure $\mu$ on $\mathbb{C}_{1/2}$ is called a \emph{Carleson measure} for $X$ provided there is a constant $C=C(X,\mu)$ such that for every $f \in X$,
\[\int_{\mathbb{C}_{1/2}}|f(s)|^2\,d\mu(s) \leq C \|f\|_X^2.\]
The smallest such constant $C(X,\mu)$ is called the \emph{Carleson constant} for $\mu$ with respect to $X$. A Carleson measure $\mu$ is said to be a \emph{vanishing Carleson measure} for $X$ provided
\[\lim_{k \to \infty} \int_{\mathbb{C}_{1/2}} |f_k(s)|^2\,d\mu(s) = 0\]
for every weakly compact sequence $\left\{f_k\right\}_{k\geq1}$ in $X$. In this case, weakly compact means that $\phi(f_k)\to0$ for every $\phi\in X^\ast$. Since both $X=D_\beta(\mathbb{C}_{1/2})$ and $X=\mathcal{D}_\alpha$ are reproducing kernel spaces, it is clear that $\left\{f_k\right\}_{k\geq1}$ in $X$ is weakly compact if and only if $\|f_k\|_X \leq C$ and $f_k(s) \to 0$ on every compact subset $K$ of $\mathbb{C}_{1/2}$.
\begin{lem}
	\label{lem:Carlesonmeasures} Let $\alpha\geq0$. Suppose that $\mu$ is a Borel measure on $\mathbb{C}_{1/2}$ with bounded support. Then $\mu$ is a Carleson measure for $\mathcal{D}_\alpha$ if and only if $\mu$ is a Carleson measure for $D_{2^\alpha-1}(\mathbb{C}_{1/2})$. Moreover, $\mu$ is vanishing Carleson for $\mathcal{D}_\alpha$ if and only if $\mu$ is vanishing Carleson for $D_{2^\alpha-1}(\mathbb{C}_{1/2})$. 
\end{lem}

The first part of this result can be extracted from \cite{Olsen11,OS12}. In preparation for the part regarding vanishing Carleson measures, let us collect some preliminary results. The following geometric characterization of Carleson measures for Bergman spaces can be found in \cite[Sec.~7.2]{Zhu}.
\begin{lem}
	\label{lem:carlesonclassic} Let $\beta\geq0$ and let $\mu$ be a Borel measure on $\mathbb{C}_{1/2}$. Then $\mu$ is a Carleson measure for $D_{\beta}(\mathbb{C}_{1/2})$ if and only if
	\[\mu\big(Q(\tau,\varepsilon)\big) = \mathcal{O}\big(\varepsilon^{\beta+1}\big)\]
	for every Carleson square $Q(\tau,\varepsilon) = [1/2,1/2+\epsilon]\times[\tau-\varepsilon/2,\tau+\varepsilon/2]$. Additionally, $\mu$ is vanishing Carleson for $D_{\beta}(\mathbb{C}_{1/2})$ if and only if
	\[\mu\big(Q(\tau,\varepsilon)\big) = o\big(\varepsilon^{\beta+1}\big),\]
	as $\varepsilon \to 0^+$, uniformly for $\tau\in\mathbb R$. 
\end{lem}

The reproducing kernels of $\mathcal{D}_\alpha$ are given by $K_\alpha(s,w) = \zeta_\alpha\left(s+\overline{w}\right)$, where
\[\zeta_\alpha(s) = \sum_{n=1}^\infty \left[d(n)\right]^\alpha n^{-s}.\]
It is clear that $\left\|K_\alpha(\cdot,w)\right\|_{\mathcal{D}_\alpha} = \sqrt{\zeta_\alpha(2\mre w)}$. We extract from \cite[pp.~240--241]{Wilson23} that 
\begin{equation}
	\label{eq:wilsonfactor} \zeta_\alpha(s) := \sum_{n=1}^\infty [d(n)]^\alpha \, n^{-s} = \left[\zeta(s)\right]^{2^\alpha} \prod_{j=1}^\infty \left(1 + \sum_{m=2}^\infty b_m p_j^{-ms}\right) =: \left[\zeta(s)\right]^{2^\alpha}\phi_\alpha(s), 
\end{equation}
where the Euler product $\phi_\alpha(s)$ converges absolutely in $\mathbb{C}_{1/2}$ with $\phi_\alpha(1)\neq0$.
\begin{proof}
	[Proof of Lemma~\ref{lem:Carlesonmeasures}] As stated above, the first part regarding Carleson measures can be extracted from \cite{Olsen11,OS12}. We will only consider the part pertaining to vanishing Carleson measures here.
	
	We argue first by contradiction. Assume that $\mu$ is vanishing Carleson for $\mathcal{D}_\alpha$, and that $\mu$ is not vanishing Carleson for $D_{2^\alpha-1}(\CC_{1/2})$. By Lemma~\ref{lem:carlesonclassic}, the latter assumption implies that there is some sequence of Carleson squares $\left\{Q_k(\tau_k,\varepsilon_k)\right\}_{k\geq1}$, where $\varepsilon_k \to 0$, satisfying
	\[\mu(Q_k) \gg \varepsilon_k^{2^\alpha}.\]
	Let $s_k = 1/2+\varepsilon_k + i\tau_k$ and consider
	\[f_k(s) = \frac{K_\alpha(s,s_k)}{\|K_\alpha(\cdot,s_k)\|_{\mathcal{D}_\alpha}} = \frac{\zeta_\alpha\left(s + \overline{s_k}\right)}{\sqrt{\zeta_\alpha(1+2\varepsilon_k)}}.\]
	It is easy to see that $f_k$ is weakly compact in $\mathcal{D}_\alpha$, since $\|f_k\|_{\mathcal{D}_\alpha}=1$ and $f_k(s)\to 0$ uniformly in $\sigma\geq1/2+\delta$ for every $\delta>0$. Since $\mu$ is assumed to be vanishing Carleson for $\mathcal{D}_\alpha$, this means that
	\[\lim_{k \to \infty }\int_{Q_k} |f_k(s)|^2\,d\mu(s) \leq \lim_{k\to \infty}\int_{\mathbb{C}_{1/2}} |f_k(s)|^2\,d\mu(s) = 0.\]
	Now, let $s = \sigma+it \in Q_k$. Then $1/2 \leq \sigma \leq 1/2+\varepsilon_k$ and $\tau_k-\varepsilon_k/2 \leq t \leq \tau_k + \varepsilon_k/2$. Recalling the simple pole of the zeta function and using \eqref{eq:wilsonfactor}, we obtain
	\[\zeta_\alpha\left(s + \overline{s_k}\right) \asymp \left(s+\overline{s_k}-1\right)^{-2^\alpha} \gg (1+2\varepsilon_k+i\varepsilon_k/2-1)^{-2^{\alpha}} \asymp \varepsilon_k^{-2^{\alpha}}.\]
	Similarly, $\sqrt{\zeta_\alpha(1+2\varepsilon_k)}\asymp \varepsilon_k^{-2^{\alpha-1}}$. Hence, by the assumption that $\mu$ is not vanishing Carleson for $D_{2^\alpha-1}(\mathbb{C}_{1/2})$, we estimate
	\[0=\lim_{k\to \infty}\int_{Q_k} |f_k(s)|^2\,d\mu(s) \gg \lim_{k\to \infty}\mu(Q_k) \varepsilon_k^{-2^{\alpha}} \gg 1,\]
	and the desired contradiction is obtained.
	
	In the other direction, assume that $\mu$ is vanishing Carleson for $D_{2^\alpha-1}(\mathbb{C}_{1/2})$. Let $\left\{f_k\right\}_{k\geq1}$ be a weakly compact sequence in $\mathcal{D}_\alpha$. Since $\mu$ has bounded support, there is some constant $M>0$ so that 
	\begin{equation}
		\label{eq:compactmuest} \int_{\mathbb{C}_{1/2}}|f_k(s)|^2\, d\mu(s) \leq M \int_{\mathbb{C}_{1/2}} \left|\frac{f_k(s)}{(s+1/2)^{2^\alpha}}\right|^2\,d\mu(s). 
	\end{equation}
	Let $F_k(s) = f_k(s)/(s+1/2)^{2^\alpha}$. Clearly $F_k(s)\to0$ on compact subsets $K$ of $\mathbb{C}_{1/2}$ since this is true for $f_k$. From \eqref{eq:DBi} and the discussion following Theorem~\ref{thm:Dequiv}, we conclude that $\|F_k\|_{D_{2^\alpha-1}} \ll \|f_k\|_{\mathcal{D}_\alpha}$. In particular, this implies that $\left\{F_k\right\}_{k\geq1}$ is a weakly compact sequence in $D_{2^\alpha-1}(\mathbb{C}_{1/2})$ and hence by \eqref{eq:compactmuest}, the measure $\mu$ is vanishing Carleson for $\mathcal{D}_\alpha$.
\end{proof}
\begin{rem}
The first part of the proof of Lemma~\ref{lem:Carlesonmeasures} does not use that $\mu$ has bounded support, so a vanishing Carleson measure for $\mathcal D_\alpha$ is always vanishing Carleson for $D_{2^\alpha-1}(\CC_{1/2})$.
\end{rem}

\subsection{Carleson measures on the polydisc} Let $\varphi \in \mathcal{G}$ with $\charac(\varphi)=0$, and let $\Phi$ denote the Bohr lift of $\varphi$. For $\beta\geq0$ we will consider the following measures on $\mathbb{C}_{1/2}$. 
\[\mu_{\beta,\varphi}(E) = 
	\begin{cases}
		\nu_\beta\big(\left\{z \in \mathbb{D}^\infty \,\colon\, \Phi(z) \in E\right\}\big), & \text{if }\beta > 0, \\
		\nu_\beta\big(\left\{z \in \mathbb{T}^\infty \,\colon\, \Phi(z) \in E \right\}\big), & \text{if }\beta = 0, 
	\end{cases}
	\qquad\qquad E \subset \mathbb{C}_{1/2}. \]
The following necessary and sufficient Carleson conditions for boundedness and compactness of $\mathcal{C}_\varphi$ when $\varphi\in\mathcal{G}$ with $\charac(\varphi)=0$ and $\varphi(\mathbb{C}_0)$ is a bounded set will be our main technical tool for the study of composition operators between the spaces $\mathcal D_\alpha$.
\begin{lem}
	\label{lem:carlesoncond} Let $\alpha,\beta\geq0$. Suppose that $\varphi\in\mathcal{G}$ with $\charac(\varphi)=0$ and suppose that $\varphi(\mathbb{C}_0)$ is a bounded subset of $\mathbb{C}_{1/2}$. Then $\mathcal{C}_\varphi \colon \mathcal{D}_\alpha \to \mathcal{D}_\beta$ is bounded if and only if 
	\begin{equation}
		\label{eq:Carlesonest} \mu_{\beta,\varphi}\big(Q(\tau,\varepsilon)\big) = \mathcal{O}\big(\varepsilon^{2^\alpha}\big) 
	\end{equation}
	for every Carleson square $Q(\tau,\varepsilon) = [1/2,1/2+\epsilon]\times[\tau-\varepsilon/2,\tau+\varepsilon/2]$. Moreover, $\mathcal{C}_\varphi$ is compact from $\mathcal{D}_\alpha$ to $\mathcal{D}_\beta$ if and only if 
	\[\mu_{\beta,\varphi}\big(Q(\tau,\varepsilon)\big) = o\big(\varepsilon^{2^\alpha}\big),\]
	as $\varepsilon \to 0^+$, uniformly for $\tau\in\mathbb R$. 
\end{lem}
\begin{proof}
	We begin with the proof of the boundedness criterion \eqref{eq:Carlesonest}. Assume at first that $\alpha,\beta>0$. Let $P$ be a Dirichlet polynomial. Since $c_0=0$, we observe as in the proof of Theorem \ref{thm:Hp} that $\mathcal B(P\circ\varphi)=P\circ\mathcal B\varphi$, so
	\begin{equation}
		\label{eq:Dalphaphi} \|\mathcal{C}_\varphi P\|_\beta^2 = \int_{\mathbb{D}^\infty} |P\left(\Phi(z)\right)|^2\,d\nu_\beta(z). 
	\end{equation}
	Now, since $\mu_{\beta,\varphi} = \nu_{\beta,\varphi}\circ \Phi^{-1}$ and since Dirichlet polynomials are dense in $\mathcal D_\alpha$, it is easy to deduce from \eqref{eq:Dalphaphi} that $\mathcal{C}_\varphi$ is bounded from $\mathcal{D}_\alpha$ to $\mathcal{D}_\beta$ if and only if
	\[\int_{\mathbb{C}_{1/2}} |f(s)|^2\,d\mu_{\beta,\varphi}(s) \ll \|f\|_{\mathcal{D}_\alpha}^2.\]
	Using Kronecker's theorem and the maximum modulus principle on the polydisc, we find that $\supp\big(\mu_{\beta,\varphi}\big) = \overline{\varphi(\mathbb{C}_0)}$. By assumption, $\varphi(\mathbb{C}_0)$ is a bounded subset of $\mathbb{C}_{1/2}$, so $\mu_{\beta,\varphi}$ has bounded support. Hence, by Lemma~\ref{lem:Carlesonmeasures} and Lemma~\ref{lem:carlesonclassic}, $\mu_{\beta,\varphi}$ is a Carleson measure for $\mathcal{D}_\alpha$ if and only if
	\[\mu_{\beta,\varphi}\big(Q(\tau,\varepsilon)\big) = \mathcal{O}\left(\varepsilon^{2^\alpha}\right).\]
	The argument for compactness follows by similar considerations. If $\alpha=0$, these arguments work line for line. If $\beta=0$, we appeal directly to \cite[Lem.~4.1]{QS15}. Clearly $\supp\big(\mu_{\beta,\varphi}\big) \subseteq \overline{\varphi(\mathbb{C}_0)}$, so the measure is still boundedly supported. The remaining deliberations apply directly. 
\end{proof}

This lemma can be combined with a compactness argument as in \cite[Lem.~6]{BB15}, to obtain the next result. But first, note that if $\varphi\in\mathcal{G}$ is a Dirichlet polynomial with $\charac(\varphi)=0$, its Bohr lift $\Phi=\mathcal{B}\varphi$ is always a polynomial of $d<\infty$ variables. We call $d$ the \emph{complex dimension} of $\varphi$ and write $d=\dim(\varphi)$.
\begin{cor}
	\label{cor:carleson} Let $\varphi\in\mathcal G$ be a Dirichlet polynomial with $\dim(\varphi)=d$ and Bohr lift $\Phi$. If for every $w\in\TT^d$ with $\mre\Phi(w)=1/2$ there exist a neighborhood $\mathcal U_w\ni w$ in $\overline{\DD^d}$, constants $C_w>0$ and $\kappa_w\geq2^\alpha$ such that, for every $\tau\in\RR$ and every $\veps>0$, $$\nu_\beta\big(\{z\in\mathcal U_w\colon \Phi(z)\in Q(\tau,\veps)\}\big)\leq C_w \veps^{\kappa_w},$$ then $\mathcal{C}_\varphi$ maps $\mathcal D_\alpha$ boundedly into $\mathcal D_\beta$. If moreover $\kappa_w>2^\alpha$ for every $w\in\TT^d$ with $\mre\Phi(w)=1/2$, then $\mathcal{C}_\varphi\colon \mathcal{D}_\alpha\to\mathcal{D}_\beta$ is compact. 
\end{cor}

\subsection{Measures of some sets in $\DD^d$} Corollary~\ref{cor:carleson} indicates that we need to estimate the measure of some sets in $\DD^{d}$. Let us collect some estimates for some particular subsets of $\mathbb{D}^d$. To simplify the computations, we will replace the measure $\nu_\beta$ with the new measure $\widetilde{\nu}_\beta$ associated to $\widetilde{m}_\beta$ as defined in \eqref{eq:mbeta}. Now, if $\dim(f)=d$, then clearly
\[\int_{\mathbb{D}^d} |\mathcal{B}f(z)|^2\,d\widetilde{\nu}_\beta(z) \asymp_{d,\beta} \int_{\mathbb{D}^d} |\mathcal{B}f(z)|^2\,d\nu_\beta(z).\]
In particular, we can replace $\nu_\beta$ by $\widetilde{\nu}_\beta$ in Corollary \ref{cor:carleson}. We should also point out that for $\beta=0$, we do not change the measure and adopt the convention $\nu_0 = \widetilde{\nu_0}$. 

For $\delta,\veps>0$, let $S(\delta,\varepsilon)=\left\{z=(1-\rho)e^{i\theta}\in\DD\,:\, 0\leq\rho\leq\delta,\, |\theta|\leq\varepsilon\right\}.$ As usual, $B(w,r)$ will denote the open ball centered at $w\in\mathbb{C}$ with radius $r>0$. Geometric considerations show that there exist absolute constants $c,C>0$ such that, for every $\veps>0$ and every $w\in\TT$, we have 
\begin{align}
	S\left(c\veps,c\veps^{1/2}\right)&\subset \{z\in\DD\,:\, \mre(1-z)<\veps\} \subset S\left(C\veps,C\veps^{1/2}\right)	 \label{eq:carlesonwindow2} \\
	wS(c\veps,c\veps)&\subset B(w,\veps)\cap\DD\subset wS(C\veps,C\veps). \label{eq:carlesonwindow1} 
\end{align}
The following lemmas are inspired by \cite{Bayart11}, and for the sake of clarity we include a brief account of their proofs. 
\begin{lem}
	\label{lem:volume0} For any $\beta>0$, $\widetilde{m}_\beta\big(S(\delta,\veps)\big)\asymp_\beta \delta^{\beta}\veps$. 
\end{lem}
\begin{proof}
	This follows from an integration in polar coordinates. 
\end{proof}
\begin{lem}
	\label{lem:volume1} For any $\beta>0$, $\widetilde{m}_\beta\big(\{z\in\DD\,:\, \mre(1-z)<\veps\}\big)\asymp_\beta\veps^{\beta+\frac 12}.$ 
\end{lem}
\begin{proof}
	The result follows from Lemma \ref{lem:volume0} and (\ref{eq:carlesonwindow2}). 
\end{proof}
\begin{lem}
	\label{lem:volume2} Let $\beta>0$ and $v\in\CC$. Then
	\[\widetilde{m}_\beta\big( \{z\in\DD \,:\, \mre(1-z)<\veps,\,|\mim(v-z)|<\veps\}\big)\ll_\beta \veps^{1+\beta}.\]
\end{lem}
\begin{proof}
	This follows again from an integration in polar coordinates.
\end{proof}
\begin{lem}
	\label{lem:volume3} Let $\beta>0$. There exists $c>0$ such that, for any $v\in\CC$ satisfying
	\[|\mre(v)-1|\leq c\veps\qquad\text{and}\qquad|\mim(v)|\leq (c\veps)^{1/2},\]
	then
	\[\widetilde{m}_\beta \big( \{z\in\DD\,:\, \mre(1-z)<\veps,\,|v-z|<\veps\}\big)\asymp_\beta \veps^{1+\beta}.\]
\end{lem}
\begin{proof}
	The upper bound is Lemma \ref{lem:volume2}. For the lower bound, observe that, provided $c\in(0,1/2)$, then $\{z\in\DD\,:\, |z-v|<\veps/2\}\subset \{z\in\DD\,:\, \mre(1-z)<\veps\}$. Hence, we just need to minorize $\widetilde{m}_\beta\big(B(v,\veps/2)\cap\mathbb D\big)$. Now, it is easy to check that upon the conditions $c\in(0,1/2)$ and $\veps\in(0,1)$, $$-8c\veps\leq 1-|v|\leq 8c\veps.$$ Writing
	\[|z-v|\leq\left|z-\frac{v}{|v|}\right|+\big|1-|v|\big|\]
	we get that $B(v/|v|,\veps/4)\subset B(v,\veps/2)$ provided $c<1/32$. We finish the proof as in Lemma \ref{lem:volume2}. 
\end{proof}
\begin{rem}
	When $\delta=\veps$, the sets $S(\delta,\veps)$ are the classical Carleson windows of the disc. However, we are required to handle inhomogeneous Carleson windows in what follows. 
\end{rem}

\section{Composition operators with polynomial symbols on $\mathcal{D}_\alpha$} \label{sec:poly} 
Let us consider a polynomial symbol in $\mathcal{G}$ of characteristic $c_0=0$, say $\varphi(s)=\sum_{n=1}^N c_n n^{-s}$. We are only interested in symbols having \emph{unrestricted range}, which means that $\varphi(\mathbb C_0)$ is not contained in $\mathbb C_{1/2+\delta}$, for any $\delta>0$. If the symbol has restricted range, it is trivial to deduce from \cite[Thm.~1]{BB14} that $\mathcal{C}_\varphi$ maps $\mathcal{D}_\alpha$ compactly into $\mathcal{D}_\beta$, for any choice of $\alpha,\beta\geq0$.

Let us now look at the Bohr lift of $\varphi$, denoted $\Phi$. As in the previous section, we will let $\dim(\varphi)$ denote the \emph{complex dimension} of $\varphi$, which is equal to the number of variables in the polynomial $\Phi(z_1,\,\ldots,\,z_d)$. Now, the \emph{degree} of $\varphi$ will be the degree of $\Phi$, and we will write $\deg(\varphi)$. When the complex dimension is big and the degree is small, we can improve $\beta=2^\alpha-1$ from the main result of \cite{BB14} substantially.
\begin{thm}
	\label{thm:dirichletpolynomial} Fix $\alpha>0$ and consider a Dirichlet polynomial $\varphi$ in $\mathcal G$ with unrestricted range. 
	\begin{itemize}
		\item[(i)] If $d=\dim(\varphi)\geq2$ and $\deg(\varphi)\in\{1,2\}$, then $\mathcal{C}_\varphi$ maps $\mathcal{D}_\alpha$ boundedly into $\mathcal D_\beta$ for some $\beta<2^\alpha-1$. More precisely, $\mathcal{C}_\varphi \colon \mathcal D_\alpha \to \mathcal D_{(2^\alpha-1)/d}$ is bounded. 
	\end{itemize}
	The result is optimal in the following sense. 
	\begin{itemize}
		\item[(ii)] If $\dim(\varphi)=1$, then $\mathcal{C}_\varphi\colon\mathcal{D}_\alpha\to\mathcal{D}_\beta$ is not bounded for any $\beta<2^{\alpha}-1$. 
		\item[(iii)] There are polynomials $\varphi\in\mathcal G$ of any complex dimension and with arbitrary $\deg(\varphi)\geq 3$ for which $\mathcal{C}_\varphi$ is not bounded from $\mathcal{D}_\alpha$ to $\mathcal{D}_\beta$ for any $\beta<2^{\alpha}-1$. 
	\end{itemize}
\end{thm}

From the proof of Theorem~\ref{thm:dirichletpolynomial} (and Corollary~\ref{cor:carleson}) it is possible to deduce the following result regarding compactness. However, before we state the result, let us stress that the inclusion $\mathcal{D}_\alpha\subset\mathcal{D}_\beta$ is \emph{not} compact for $\alpha<\beta$. To realize this one needs only consider the weakly compact sequence generated by the prime numbers, $\{p_j^{-s}\}_{j\geq1}$, since $d(p_j)=2$.
\begin{cor}
	\label{cor:compact} Fix $\alpha>0$ and consider a Dirichlet polynomial $\varphi$ in $\mathcal G$ with unrestricted range.
	\begin{itemize}
		\item[(i)] If $\dim(\varphi)\geq2$ and $\deg(\varphi)\in\{1,2\}$, then $\mathcal{C}_\varphi\colon\mathcal{D}_\alpha\to\mathcal{D}_{2^\alpha-1}$ is compact. 
	\end{itemize}
	The result is optimal in the following sense. 
	\begin{itemize}
		\item[(ii)] If $\dim(\varphi)=1$, then $\mathcal{C}_\varphi\colon\mathcal{D}_\alpha\to\mathcal{D}_{2^\alpha-1}$ is never compact. 
		\item[(iii)] There are polynomials $\varphi\in\mathcal G$ of any complex dimension and with arbitrary $\deg(\varphi)\geq 3$ for which $\mathcal{C}_\varphi\colon\mathcal{D}_\alpha\to\mathcal{D}_{2^\alpha-1}$ is not compact. 
	\end{itemize}
\end{cor}

It is interesting to compare Corollary~\ref{cor:compact} to its version for $\alpha=0$ which is \cite[Thm.~3]{BB15}. Ignoring the technical part of \cite[Thm.~3]{BB15} regarding minimal Bohr lift and boundary index, we observe that the results match up. However, going into the details, we observe that this correspondance is not completely true. We shall give later (see Theorem~\ref{thm:linn}) simple examples of polynomial symbols $\varphi$ such that $\mathcal C_\varphi$ maps $\mathcal{D}_\alpha$ compactly into $\mathcal{D}_{2^\alpha-1}$ for $\alpha>0$, but does not map $\mathcal H^2$ compactly into $\mathcal H^2$. This phenomenon is due to the necessity to introduce the minimal Bohr lift in the context of $\mathcal H^2$.

Observe also that it is possible to deduce a version of Theorem~\ref{thm:dirichletpolynomial} for the case $\mathcal{C}_\varphi\colon\mathcal{D}_\alpha\to\mathcal{H}^2$ from \cite[Lem.~10]{BB15} using Lemma~\ref{lem:carlesoncond} and Corollary~\ref{cor:carleson}. However, the result would be cumbersome to state, due to the above mentioned technical parts, so we avoid it here.

We need one final lemma to prove Theorem \ref{thm:dirichletpolynomial}, which can easily be deduced from the Julia--Caratheodory theorem (or from elementary considerations as in the proof of \cite[Lem.~7]{BB15}).
\begin{lem}
	\label{lem:coefficients} Let $P(z) = \sum_{k=1}^K a_k(1-z)^k$ be a polynomial mapping $\DD$ into $\overline{\CC_0}$. Then $P\equiv0$ or $a_1>0$. 
\end{lem}
We split the proof of Theorem~\ref{thm:dirichletpolynomial} into two parts, and begin with the easiest part.
\begin{proof}[Proof of Theorem~\ref{thm:dirichletpolynomial} --- (ii) and (iii)]
	We begin with (ii). Fix $\alpha>0$ and assume that $\varphi \in \mathcal{G}$ is a Dirichlet polynomial with $\dim(\varphi)=1$ and unrestricted range. By Corollary~\ref{cor:carleson} we investigate some $w \in \mathbb{T}$ such that $\Phi(w)=1/2+i\tau$, where $\Phi$ denotes the Bohr lift of $\varphi$. We may assume that $w = 1$ and $\tau=0$ after, if necessary, a (complex) rotation and a (vertical) translation. Hence, $\Phi$ is a polynomial of the form
	\[\Phi(z) = \frac{1}{2} + \sum_{k=1}^K a_k(1-z)^k.\]
	By Lemma~\ref{lem:coefficients} we know that $a_1>0$. In view of Corollary~\ref{cor:carleson}, it suffices to prove that for $\beta>0$ and every small enough $\varepsilon>0$, 
	\[\mu_{\beta,\varphi}\big(Q(0,\varepsilon)\big)\gg \varepsilon^{\beta+1}.\]
 Using Lemma~\ref{lem:volume0}, we see that it is sufficient to prove that the homogeneous Carleson window $S(\veps,\veps)$ is included in the pre-image of $Q(0,c\varepsilon)$ under $\Phi$ for some fixed $c\in(0,1)$ and 
	for every small enough $\varepsilon>0$. Now, note that if $z \in S(\varepsilon,\varepsilon)$, then
	\[\max\left\{\mre\big((1-z)^k\big),\,\mim\big((1-z)^k\big)\right\} \leq \varepsilon^k.\]
	In particular, since $\Phi$ is a polynomial and $a_1>0$, we find that if $z \in S(\varepsilon,\varepsilon)$, then
	\begin{align*}
		1/2 \leq \mre \Phi(z) &\leq 1/2 + a_1 \varepsilon + \mathcal{O}(\varepsilon^2), \\
		|\mim \Phi(z)| &\leq a_1 \varepsilon + \mathcal{O}(\varepsilon^2).
	\end{align*}
	Hence any $c>a_1/2$ will do. Part (iii) can be deduced from this argument in the following way. Let $\delta>0$ and let $\Psi(z)=\Psi(z_1,\,\ldots,\,z_d)$ be any polynomial in $d$ variables and define 
	\[\Phi(z)=\frac{1}{2}+(1-z_1)+\delta(1-z_1)^2\Psi(z).\] 
	Clearly $\Phi$ is the Bohr lift of 
	\[\varphi(s)=\frac 12+(1-p_1^{-s})+\delta(1-p_1^{-s})^2\Psi(p_1^{-s},\,\ldots,\,p_d^{-s}).\]
	It is proved in \cite[Lem.~9]{BB15} that by choosing $\delta>0$ sufficiently small, we can guarantee that $\varphi\in\mathcal{G}$, that $\varphi$ has unrestricted range and furthermore that if $\Phi$ touches the boundary of $\mathbb{C}_{1/2}$ at some point $z \in \overline{\mathbb{D}^d}$, then necessarily $z_1=1$. The argument given above works line for line with one minor modification. Suppose $z_1 \in S(\varepsilon,\varepsilon)$. Then for every choice of $z_2,\,\ldots,\,z_d$ in $\overline{\mathbb{D}}$ we have
	\[\max\left\{\mre\big(\delta(1-z_1)^2\Psi(z)\big),\,\mim\big(\delta(1-z_1)^2\Psi(z)\big)\right\}\leq\delta\|\Psi\|_\infty \varepsilon^2,\]
	so we conclude again by Corollary~\ref{cor:carleson} and Lemma~\ref{lem:volume0}.
\end{proof}

\begin{proof}
	[Proof of Theorem \ref{thm:dirichletpolynomial} --- (i)] Let $\varphi\in\mathcal G$ be a Dirichlet polynomial and assume that $\dim(\varphi)=d\geq 2$ and $\deg(\varphi)\in\{1,2\}$. Let $\Phi$ be the Bohr lift of $\varphi$. We will again apply Corollary \ref{cor:carleson}. Hence, let $w\in\TT^d$ be such that $\mre\Phi(w)=1/2$. Without loss of generality, we may assume that $w=\mathbf 1=(1,\dots,1)$ and that $\Phi(\mathbf 1)=1/2$. We may write $\Phi$ as 
	\[\Phi(z)=\frac12+\sum_{j=1}^d a_j (1-z_j)+\sum_{j=1}^d b_j (1-z_j)^2+\sum_{1\leq j<k\leq d}c_{j,k}(1-z_j)(1-z_k).\]
	We first claim that $a_j>0$ for any $j=1,\,\ldots,\,d$. Indeed, applying Lemma \ref{lem:coefficients} to $\Phi(\mathbf 1,z_j,\mathbf 1)-1/2$, we know that either $a_j>0$ or $a_j=b_j=0$. Assume that the latter case holds. Since $\varphi$ has complex dimension $d$, there exists $k\neq j$ so that $c_{j,k}\neq 0$. Let us consider $\Psi(z_j,z_k)=\Phi(\mathbf 1,z_j,\mathbf 1,z_k,\mathbf 1)$. Then a Taylor expansion of $\Psi(e^{i\theta_j},e^{i\theta_k})$ shows that 
	\[\mre \Psi\big(e^{i\theta_j},e^{i\theta_k}\big)=\frac 12+\left(\frac{a_k}2-\mre(b_k)\right)\theta_k^2-\mre(c_{j,k})\theta_j\theta_k+o(\theta_j^2)+o(\theta_k^2).\] 
	Choosing $\theta_j=\delta$ and $\theta_k=\delta^2$ and letting $\delta$ to $0$, this implies that $\mre(c_{j,k})=0$ since by assumption $\mre \Psi\geq 1/2$. On the other hand, for $\rho_j\in(0,1)$,
\[\mre \Psi(1-\rho_j,e^{i\theta_k})=\frac 12+ \left(\frac{a_k}2-\mre (b_k)\right)\theta_k^2+\mim(c_{j,k})\rho_j\theta_k+o(\rho_j^2)+o(\theta_k^2).\]
This in turn yields that $\mim (c_{j,k})=0$, a contradiction.
	
	We come back to $\Phi$ and, for $j=1,\,\ldots,\,d$, we write $z_j=(1-\rho_j)e^{i\theta_j}$ where $\rho_j\in(0,1)$ and $\theta_j\in [-\pi,\pi)$. We shall use the local diffeomorphism between a neighborhood of $\bf 1$ in $\CC^d$ and a neighborhood of $\bf 0$ in $\RR^{2d}$ given by
	\[(\rho,\theta)\mapsto\big((1-\rho_1)e^{i\theta_1},\,\ldots,\,(1-\rho_d)e^{i\theta_d}\big).\] 
	A Taylor expansion of $\mre\Phi$ near $\mathbf 1$ shows that 
	\[\mre\Phi(z)=\frac12+\sum_{j=1}^d \rho_j F_j(\rho,\theta)+G(\theta)\]
	where $F_j(\mathbf 0)=a_j$. Taking all $\rho_j$ equal to zero, we get that $G(\theta)\geq 0$. Hence, there exists a (fixed) neighborhood $\mathcal{U}\ni\mathbf 1$ in $\overline{\DD^d}$ such that for all $\varepsilon>0$ and all $\tau\in\mathbb{R}$, 
	\[0\leq \sum_{j=1}^d \rho_j F_j(\rho,\theta)\leq \veps\qquad \text{and}\qquad F_j(\rho,\theta)\geq \frac{a_j}2\]
	provided $z\in\mathcal U$ and $\Phi(z)\in Q(\tau,\veps)$. This implies that $|\rho_j|\leq 2\veps/a_j$ for any $j=1,\ldots,d$. We now look at $\mim\Phi$ and let us write it under the following form:
	\[\mim \Phi(z)=\gamma_{(\rho,\theta_2,\ldots,\theta_d)}(\theta_1)=a_1\theta_1+o(\theta_1).\]
	The map $(\rho,\theta)\mapsto \gamma_{(\rho,\theta_2,\ldots,\theta_d)}(\theta_1)$ is smooth and satisfies $\gamma^\prime_{(\rho,\theta_2,\ldots,\theta_d)}(0)=a_1$. Then there exists $\mathcal V$ a neighborhood of $\mathbf 1$ in $\CC^d=\RR^{2d}$ such that, for any $z\in\mathcal V$, 
	\begin{eqnarray}
		\label{eq:derivative} \gamma^\prime_{(\rho,\theta_2,\ldots,\theta_d)}(\theta_1)\geq \frac{a_1}{2}. 
	\end{eqnarray}
	Now, if $(\rho,\theta_2,\ldots,\theta_d)$ are fixed and $\theta_1$ is such that $z$ belongs to $\mathcal V$, the condition $\Phi(z)\in Q(\tau,\veps)$ implies that $\gamma_{(\rho,\theta_2,\ldots,\theta_d)}(\theta_1)$ belongs to some interval of length $\veps$. By (\ref{eq:derivative}), this implies that $\theta_1$ belongs to some interval of length $C\veps$, where $C$ does not depend on $(\rho,\theta_2,\dots,\theta_d)$ provided $z\in \mathcal V$.
	
	Let us summarize the previous computations. We have shown that there exist a (fixed) neighborhood $\mathcal W=\mathcal U\cap\mathcal V$ of $\mathbf 1$ in $\CC^d$ and a constant $D>0$ such that, for any $z\in\mathcal W\cap\DD^d$ and any $\veps>0$ satisfying $\Phi(z)\in Q(\tau,\veps)$, then $\rho_j\leq D\veps$ and $\rho,\,\theta_2,\,\ldots,\,\theta_d$ being fixed, $\theta_1$ belongs to some fixed interval of length $D\veps$. By Fubini's theorem and polar integration as in Lemma \ref{lem:volume0}, we get that 
	\[\widetilde{\nu}_\beta\left(\big\{z\in \mathcal W\cap \DD^d\,:\, \Phi(z)\in Q(\tau,\veps)\big\}\right)\ll \veps^{d\beta+1}.\]
	We conclude by Corollary \ref{cor:carleson}.
\end{proof}

Let us focus our attention on part (ii) of Theorem~\ref{thm:dirichletpolynomial}, which implies that it is sufficient to consider the most simple non-trivial symbol,
\begin{equation} \label{eq:simplesymbol}
	\varphi(s) = 3/2 - 2^{-s} = 1/2 + (1-2^{-s}),
\end{equation}
to conclude that the sharp $\beta$ for $\mathcal{C}_\varphi\colon\mathcal{D}_\alpha\to\mathcal{D}_\beta$ is $\beta=2^\alpha-1$. This is perhaps not so surprising, since we can consider \eqref{eq:simplesymbol} a local version of the symbol associated to the transference map,
\[\mathcal{T}(2^{-s}) = \frac{1}{2} + \frac{1-2^{-s}}{1+2^{-s}},\]
as considered in Section~\ref{sec:bohr}. We will devote the remainder of this section to investigating two classes of examples that generalize \eqref{eq:simplesymbol}.

The first extension of \eqref{eq:simplesymbol} are the \emph{linear symbols}, namely symbols which are of the form 
\begin{equation}
	\label{eq:linear} \varphi(s) = c_1 + \sum_{j=1}^d c_{p_j}p_j^{-s}. 
\end{equation}
Observe in particular that \eqref{eq:simplesymbol} is just the case $d=1$. We have the following result. 
\begin{thm}
	\label{thm:linear} Let $\alpha,\beta\geq 0$. Let $\varphi$ of the form \eqref{eq:linear} with unrestricted range and $c_{p_j}\neq 0$ for every $j$. Then $\mathcal C_{\varphi}\colon\mathcal D_\alpha\to\mathcal D_\beta$ is bounded if and only if 
	\begin{equation}
		\label{eq:linab} \frac 12+d\left(\frac 12+\beta\right)\geq 2^\alpha. 
	\end{equation}
	Moreover, $\mathcal C_{\varphi}\colon\mathcal D_\alpha\to\mathcal D_\beta$ is compact if and only if the inequality in \eqref{eq:linab} is strict. 
\end{thm}
\begin{proof}
	If $\beta=0$, this can be extracted from \cite[Lem.~8.2]{QS15} in combination with Corollary~\ref{cor:carleson}. 
	
	Assume therefore that $\beta>0$. Arguing as in \cite{QS15}, we may assume that $c_1>0$ and that $c_{p_j}<0$ for every $j$. Since $\varphi$ has unrestricted range, we know that
	\[c_1=\frac{1}{2} + \sum_{j=1}^d |c_{p_j}|.\]
	We will represent the Bohr lift of $\varphi$ in the following way. 
	\begin{equation}
		\label{eq:linearbohr} \Phi(z)=c_1+\sum_{j=1}^d c_{p_j}-\sum_{j=1}^d c_{p_j}(1-z_j)=\frac 12+\sum_{j=1}^d |c_{p_j}| (1-z_j). 
	\end{equation}
	Let $\tau\in\RR$ and $\veps>0$. If $\Phi(z)\in Q(\tau,\veps)$, we inspect \eqref{eq:linearbohr} to conclude, for any $j=1,\ldots,d-1$, that
	\[\mre(1-z_j)\leq \frac{\veps}{|c_{p_j}|}.\]
	Hence, for any $j=1,\ldots,d-1$, by Lemma~\ref{lem:volume1} we know that $z_j$ belongs to some set $R_j(\veps)$ satisfying $\widetilde{m}_\beta(R_j(\veps))\ll \veps^{\frac12+\beta}$. Moreover, for a fixed value of $z_1,\ldots,z_{d-1}$, we also have
	\[\mre(1-z_d)\leq \frac{\veps}{|c_{p_d}|}\qquad\text{and}\qquad|\mim(v-z_d)|\leq\frac{\veps}{2|c_{p_d}|}\]
	for $v\in\CC$ depending on $\tau,z_1,\ldots,z_{d-1}$. By Lemma~\ref{lem:volume2}, $z_d$ belongs to some set $R_d(z_1,\ldots,z_{d-1})$ satisfying $\widetilde{m}_\beta(R_d(z_1,\ldots,z_{d-1}))\ll \veps^{1+\beta}$. Using Fubini's theorem, we get $$\mu_{\beta,\varphi}\big(Q(\tau,\veps)\big)\ll \veps^{\frac 12+d\left(\frac12+\beta\right)}$$ and we compare this with the sufficient condition for continuity.
	
	Conversely, assume that $z_j$ belongs to $D(\eta)=\{z\in\DD\,:\, \mre(1-z)\geq \eta \veps\}$ for some small $\eta>0$ and for any $j=1,\ldots,d-1$. Observe that $\widetilde{m}_\beta\big(D(\eta)\big)\gg \veps^{\frac12+\beta}$. Then, setting
	\[v=\big(|c_{p_1}|(1-z_1)+\cdots+|c_{p_{d-1}}|(1-z_{d-1})\big)/|c_{p_d}|\]
	we get 
	\begin{eqnarray*}
		\Phi(z_1,\ldots,z_d)\in Q(0,\veps)&\iff& \left\{ 
		\begin{array}{l}
			0\leq \mre\big(c_1-|c_{p_1}|z_1-\cdots-|c_{p_d}|z_d\big)\leq\veps\\
			\left|\mim\big(|c_{p_1}|z_1+\cdots+|c_{p_d}|z_d\big)\right|\leq\veps/2
		\end{array}
		\right.\\
		&\iff&\left\{ 
		\begin{array}{l}
			0\leq |c_{p_1}|\mre(1-z_1)+\cdots+|c_{p_d}|\mre(1-z_d)\leq\veps\\
			\left|\mim\big(|c_{p_1}(1-z_1)+\cdots+|c_{p_{d-1}}|(1-z_{d-1})-|c_{p_d}|z_d\big)\right|\leq\veps /2
		\end{array}
		\right.\\
		&\iff&\left\{ 
		\begin{array}{l}
			0\leq \mre\big(v+(1-z_d)\big)\leq \veps/{|c_{p_d}|}\\
			\left|\mim\big(v-z_d\big)\right|\leq\veps/{2|c_{p_d}|}. 
		\end{array}
		\right. 
	\end{eqnarray*}
	Now, $\mre(v)\leq C\eta\veps$ and $|\mim(v)|\leq (2C\eta\veps)^{1/2}$ for
	\[C=\frac{|c_{p_1}|+\cdots+|c_{p_{d-1}}|}{|c_{p_d}|}.\]
	Hence, provided $\eta$ is small enough, then $\Phi(z_1,\ldots,z_d)\in Q(0,\veps)$ as soon as $\mre(1-z_d)<\eta\veps$ and $|v-z_d|<\eta\veps$. By Fubini's theorem and Lemma \ref{lem:volume3},
	\[\mu_{\beta,\varphi}\big(Q(0,\veps)\big)\gg \veps^{\frac12+d\left(\frac 12+\beta\right)}.\]
	We conclude by Corollary~\ref{cor:carleson}. The same proof shows that $\mathcal{C}_\varphi$ maps $\mathcal{D}_\alpha$ compactly into $\mathcal{D}_\beta$ if and only if $1/2 + d(1/2+\beta)>2^\alpha$. 
\end{proof}

It is clear that in $\mathcal H^2$, the monomials $n^{-s}$ all have norm $1$. This is of course no longer the case in $\mathcal{D}_\alpha$ when $\alpha>0$. Thus we have more flexibility in choosing the Bohr lift for $\mathcal{H}^2$, since we may use any sequence of independent integers $(q_1,\ldots,q_d)$ instead of $(p_1,\ldots,p_d)$. This lead us to introduce the notion of \emph{minimal} Bohr lift in \cite{BB15}. For the Bergman spaces, we are by definition required to consider the \emph{canonical} Bohr lift, since it is used to compute the norm. In this sense the situation is less subtle. To further emphasize the difference between $\alpha=0$ and $\alpha>0$, we have the following result.
\begin{thm}
	\label{thm:linn} Let $\alpha\geq0$ and consider $\varphi(s) = 3/2 - n^{-s}$ for some fixed integer $n\geq2$. Set $d=\dim(\varphi)$, which in this case is equal to the number of distinct prime factors of $n$. Then $\mathcal{C}_\varphi\colon\mathcal{D}_\alpha\to\mathcal{D}_\beta$ is bounded if and only if $\beta\geq(2^\alpha-1)/d$. If $\alpha=0$, then $\mathcal{C}_\varphi$ is not compact on $\mathcal{H}^2$. If $\alpha>0$, then $\mathcal{C}_\varphi\colon\mathcal{D}_\alpha \to \mathcal{D}_\beta$ is compact if and only if $\beta>(2^\alpha-1)/d$. 
\end{thm}
Observe that for every $\alpha>0$, we can make $\mathcal{C}_\varphi$ map $\mathcal{D}_\alpha$ into $\mathcal{D}_\beta$ for any $\beta>0$, by increasing the number of prime factors in $n$. However, we can never obtain $\beta=0$ in this case.
\begin{proof}
	Assume first that $\alpha=0$. As explained in \cite{BB15}, the minimal Bohr lift is simply $\Phi(z)=3/2-z$ for every integer $n\geq2$, and by the results in \cite{BB15}, this means that $\mathcal{C}_\varphi\colon\mathcal{H}^2\to\mathcal{H}^2$ is bounded, but not compact.
	
	Assume now that $\alpha>0$. Let $p$ be any prime number that does not divide $n$ and consider $\psi(s) = 3/2-p^{-s}$. By Theorem~\ref{thm:dirichletpolynomial}~(ii) and Corollary~\ref{cor:compact}~(ii) we know that $\mathcal{C}_\psi\colon\mathcal{D}_\alpha \to \mathcal{D}_\beta$ is bounded if and only if $\beta\geq2^\alpha-1$ and compact if and only if $\beta>2^\alpha-1$. Now, define the operator $T$ on $\mathcal{C}_\psi(\mathcal{D}_{\alpha})$ by $T(p^{-s})=n^{-s}$ so that $\mathcal{C}_\varphi = T\circ\mathcal{C}_\psi$. A trivial estimate with the divisor function shows that if $g \in \mathcal{C}_\psi(\mathcal{D}_{\alpha})$, then
	\[\|g\|_{\mathcal{D}_\beta} \leq \|T(g)\|_{\mathcal{D}_{\beta/d}} \ll_{n,\beta} \|g\|_{\mathcal{D}_\beta},\]
	so we are done. 
\end{proof}

\begin{rem}
It is natural to ask whether the space $\mathcal D_{(2^\alpha-1)/d}$ in Theorem~\ref{thm:dirichletpolynomial}~(i) is optimal. We found that this is not the case for linear symbols in Theorem~\ref{thm:linear}. By Theorem~\ref{thm:linn}, it is optimal if $\dim(\varphi)=2$. For $\dim(\varphi)\geq 3$, we conjecture that $(2^\alpha-1)/d$ is not optimal, but our results do not further substantiate this claim.
\end{rem}

\section{Composition operators with linear symbols on $\mathcal{H}^p$} \label{sec:hp}
Let us reiterate that the results of the previous section show that the optimal $\beta$ for the local embedding of $\mathcal{D}_\alpha$ can, through the results of Section~\ref{sec:bohr}, be decided simply by considering the symbol $\varphi(s)=3/2-2^{-s}$. The embedding problem is in general open for $\mathcal{H}^p$, so it is therefore interesting to investigate how the composition operator generated by this symbol acts on $\mathcal{H}^p$.

As previously mentioned, composition operators with characteristic $0$ acting on $\mathcal{H}^p$ are not well understood when $p$ is not an even integer. In particular, very few examples are known. To our knowledge, the only known non-trivial examples appear in \cite{BQS16}. The symbols of these operators are given by
\begin{equation} \label{eq:lensmaps}
	\varphi(s)=\frac{1}{2}+\left(\frac{1-\omega(2^{-s})}{1+\omega(2^{-s})}\right)^{1-\veps}
\end{equation}
where $\omega$ is an analytic self-map of $\mathbb{D}$ and $\veps\in(0,1)$. Observe that the fact that we are not allowed to set $\varepsilon=0$ restricts the range of $\varphi$ in $\mathbb{C}_{1/2}$. Symbols of this type are a type of lens maps from $\mathbb{C}_0$ to $\mathbb{C}_{1/2}$. Observe also that the most simple case $\omega(z)=z$ yields a restricted version of the ``transference map'' from Theorem~\ref{thm:Hp} (iii).

Now, it is clear that $\varphi(s) = 3/2 -2^{-s}$, or indeed any Dirichlet polynomial, is not of the form \eqref{eq:lensmaps}. We are not able to settle the boundedness of the composition operator induced by this symbol on $\mathcal{H}^p$, but we will again consider symbols of linear type. Using Theorem \ref{thm:bergmanemb}, we will be able to prove boundedness when the complex dimension is bigger than or equal to 2.

Our last main tool for this will be the so-called $p/q$\emph{--Carleson measures}. Let $1\leq p,q<\infty$ and let $X$ be one of the spaces considered in this paper, for instance $X=\mathcal{H}^q$ or $X=H^q(\mathbb{C}_{1/2})$. If $X = \mathcal{D}_\alpha$ or $X = D_\beta(\CC_{1/2})$ then $q=2$. We require that a measure $\mu$ satisfies
\begin{equation} \label{eq:pqC}
	\left(\int_{\mathbb{C}_{1/2}}|f(s)|^p\,d\mu(s)\right)^\frac{1}{p} \leq C \|f\|_X,
\end{equation}
for some constant $C=C(p,q,X)$ to be $p/q$--Carleson for $X$. For $X=H^q(\mathbb{C}_{1/2})$ and $q\leq p$, the following description can be found in \cite[Thm.~9.4]{Duren}.
\begin{lem} \label{lem:pq1}
	Let $1\leq q \leq p <\infty$. A positive Borel measure $\mu$ on $\mathbb{C}_{1/2}$ is $p/q$--Carleson for $H^q(\mathbb{C}_{1/2})$ if and only if
	\[\mu\big(Q(\tau,\varepsilon)\big) = \mathcal{O}\big(\varepsilon^{p/q}\big)\]
	for every Carleson square $Q(\tau,\varepsilon) = [1/2,1/2+\epsilon]\times[\tau-\varepsilon/2,\tau+\varepsilon/2]$.
\end{lem}
Let us now extend a result from \cite{OS12} to the case $p< q$, which will be needed in the proof of Theorem~\ref{thm:linearHp} for the range $2<p<\infty$.

\begin{lem} \label{lem:pq2}
	Fix $1\leq q\leq p<\infty$ and let $\mu$ be a positive Borel measure on $\mathbb{C}_{1/2}$.
	\begin{itemize}
		\item[(i)] If $\mu$ is $p/q$--Carleson for $\mathcal{H}^q$, then $\mu$ is $p/q$--Carleson for $H^q(\mathbb{C}_{1/2})$.
		\item[(ii)] If the embedding \eqref{eq:localemb} holds for $q$ and $\mu$ has bounded support, then the converse is true.
	\end{itemize}
\end{lem}
\begin{proof}
	To prove part (i), we use Lemma~\ref{lem:pq1} and argue by contradiction as in the first part of the proof of Lemma~\ref{lem:Carlesonmeasures}. In particular, assume that $\mu$ is a $p/q$--Carleson measure $\mathcal{H}^q$. Consider a sequence of Carleson squares $Q_k=Q(\tau_k,\varepsilon_k)$ and the Dirichlet series
	\[f_k(s) = [\zeta(s+1/2+\varepsilon_k+i\tau_k)]^{2/q},\]
	which satisfies $\|f_k\|_{\mathcal{H}^q}=[\zeta(1+2\varepsilon_k)]^{1/q}$. We deduce from \eqref{eq:pqC} that $\mu(Q_k) \ll \varepsilon_k^{p/q}$ as $\varepsilon_k\to0$ and conclude as in Lemma~\ref{lem:Carlesonmeasures}. Part (ii) follows from a routine application of the embedding. We proceed as in the proof of the second part of Lemma~\ref{lem:Carlesonmeasures}, setting now $F(s) = f(s)/(s+1/2)^{2/q}$ and using the same trick as in \eqref{eq:compactmuest}.
\end{proof}

To make the statement of our final lemma more convenient, we will move to $\mathbb{C}_0$ as in \cite{BB15}. This change is easily carried out when working with composition operators, since it corresponds to the changes $f(s) \mapsto f(s+1/2)$ and $\varphi(s) \mapsto \varphi(s+1/2)-1/2$. In particular, the translated $f\in\mathcal{D}_\alpha$ is embedded in $D_{2^\alpha-1,\operatorname{i}}(\mathbb{C}_0)$.

Let $d_{\operatorname{H}}(z,w)$ be the hyperbolic distance in the half-plane $\CC_0$ which is defined by 
\[\frac{1-e^{-d_{\operatorname{H}}(z,w)}}{1+e^{-d_{\operatorname{H}}(z,w)}}=\left|\frac{z-w}{z+\overline w}\right|\]
and let $B_{\operatorname{H}}(s,r)$ be the hyperbolic disc of centre $s$ and radius $r\in(0,1)$. It is well-known that $B_{\operatorname{H}}(s,r)$, for $s=\sigma+it$, is simply the Euclidean disc of centre $(\sigma\cosh r,t)$ and radius $\sigma\sinh r$. In particular, we shall use that if $r$ is not too big, then $B_{\operatorname{H}}(s,r)$ is contained in $[\sigma/2,2\sigma]\times [t-\sigma,t+\sigma]$. 

Luecking in \cite{Lue93} has characterized the $p/q$--Carleson measures of the (unweighted) Bergman spaces in the unit disc when $p<q$. As observed in \cite{PZ15}, his proof carries on the weighted Bergman spaces. 
 The next lemma is simply \cite[Thm.~B]{PZ15} with $p=2$ and $n=0$, translated from $D_\beta(\mathbb{D})$ to $D_{\beta,\operatorname{i}}(\mathbb{C}_0)$ using $\mathcal{T}-1/2$.

\begin{lem} \label{lem:luecking}
Let $1\leq p< 2$ and let $\mu$ be a positive Borel measure on $\mathbb{C}_0$. Then $\mu$ is $p/2$--Carleson for $D_{\beta,\operatorname{i}}(\mathbb{C}_0)$ if and only if for some (any) $r>0$,
\begin{equation} \label{eq:lint}
	\int_{\mathbb{C}_0} \left(\mu\big(B_{\operatorname{H}}(s,r)\big)\right)^\frac{2}{2-p}\,\sigma^{\frac{(p-4)-p\beta}{2-p}}\,|s+1|^{\frac{2p(\beta+1)}{2-p}}\,dm_1(s)<\infty.
\end{equation}
\end{lem}

We are finally in a position to prove Theorem \ref{thm:linearHp}.
\begin{proof}[Proof of Theorem \ref{thm:linearHp}]
We first assume $p>2$. We begin by fixing some positive integer $k$ and consider $q=2k< p < 2k+2$. We want to investigate when $\mathcal{C}_\varphi$ maps $\mathcal{H}^q$ to $\mathcal{H}^p$. Since $p>q$, this also means that $\mathcal{C}_\varphi$ acts boundedly on $\mathcal{H}^p$. Setting $\mu_\varphi:=\mu_{0,\varphi}$, we argue as in the proof of Lemma~\ref{lem:carlesoncond} to find that boundedness of $\mathcal{C}_\varphi\colon\mathcal{H}^q\to\mathcal{H}^p$ is equivalent to 
	\begin{equation} \label{eq:pqbound}
		\left(\int_{\mathbb{C}_{1/2}}|f(s)|^p\,d\mu_\varphi(s)\right)^\frac{1}{p} \ll \|f\|_{\mathcal{H}^q}.
	\end{equation}
	Using Lemma~\ref{lem:pq1} and Lemma~\ref{lem:pq2}, while keeping in mind that the embedding \eqref{eq:localemb} holds for $q=2k$, we find that \eqref{eq:pqbound} is equivalent to
	\[\mu_\varphi\big(Q(\tau,\epsilon)\big)\ll \varepsilon^{p/q},\]
	for every Carleson square $Q$. However, from \cite[Lem.~8.2]{QS15} we know that $\mu_\varphi\big(Q(\tau,\epsilon)\big) \ll \varepsilon^{(d+1)/2}$. Hence we require of $d$ that
	\[d\geq \frac{2p}{q}-1 = \frac{p}{k}-1.\]
	It is easy to check that $d\geq2$ is sufficient if $p \in (2,3]\cup(4,\infty)$ and $d\geq3$ is sufficient if $3<p<4$.
	
	We now consider $1\leq p < 2$. First, we use Theorem~\ref{thm:linear} with $\alpha=1$ and $\beta=0$ to conclude that if $d\geq3$, then $\mathcal{C}_\varphi$ maps $\mathcal{D}_1$ boundedly into $\mathcal{D}_0=\mathcal{H}^2$. To conclude that $\mathcal{C}_\varphi\colon\mathcal{H}^p\to\mathcal{H}^p$ is bounded, we use the inequalities
	\[\|f\|_{\mathcal{D}_1} \leq \|f\|_{\mathcal{H}^1} \leq \|f\|_{\mathcal{H}^p} \leq \|f\|_{\mathcal{H}^2},\]
	where the first one is Helson's inequality.
	
	It remains to prove that $d\geq2$ is sufficient when $1\leq p < 2$ and $p\in(3,4)$. The trivial identity
\[\left\|f\circ \varphi\right\|_{\mathcal{H}^{2p}}^{2p} = \left\|f^2 \circ \varphi \right\|_{\mathcal{H}^p}^p\]
shows that it is enough to conclude for $p\in [1,2)$. 	
 Assume that $\varphi(s) = c_1 + c_{p_1} p_1^{-s} + c_{p_2} p_2^{-s}$ has unrestricted range. Using Theorem \ref{thm:bergmanemb},  we find that it is sufficient to verify that
	\[\left(\int_{\mathbb{T}^2}|f\circ\Phi(z)|^p\,d\nu(z)\right)^\frac{1}{p} = \left(\int_{\mathbb{C}_{1/2}} |f(s)|^p\, d\mu_{0,\varphi}(s)\right)^\frac{1}{p} \ll \|f\|_{D_{\frac 2p-1,\operatorname{i}}(\mathbb{C}_{1/2})}.\]
	We now move to $\mathbb{C}_0$ to use Lemma~\ref{lem:luecking}, and subtract $1/2$ from $\varphi$. Arguing as in \eqref{eq:linearbohr} we may assume that $\widetilde{\Phi}(z) = |c_{p_1}|(1-z_1) + |c_{p_2}|(1-z_2)$, and we consider the measure $\widetilde{\mu}$ defined on $\mathbb{C}_0$ by
	\[\widetilde{\mu}(E) = \nu\big(\big\{(z_1,z_2)\in\mathbb{T}^2\,:\, \widetilde{\Phi}(z_1,z_2)\in E\big\}\big), \qquad \qquad E \subset \mathbb{C}_0.\]
	We need to investigate for which $1\leq p<2$ the measure $\widetilde{\mu}$ satisfies the condition of Lemma~\ref{lem:luecking} with 
$\beta=2/p-1$. Recall that, for $s=\sigma+it$ and some suitably small $r>0$,
\begin{eqnarray*}
\widetilde{\mu}\big(B_{\operatorname{H}}(s,r)\big)&\leq&\widetilde{\mu}\big([\sigma/2,2\sigma]\times [t-\sigma,t+\sigma]\big)\\
&=&\nu\left(\left\{(z_1,z_2)\in\mathbb{T}^2\,:\, |c_{p_1}|(1-z_1)+|c_{p_2}|(1-z_2)\in [\sigma/2,2\sigma]\times [t-\sigma,t+\sigma]\right\}\right).
\end{eqnarray*}
	Since $\widetilde{\Phi}(\mathbb{T}^2)$ is a bounded subset of $\overline{\mathbb{C}_0}$, it is clear that $\widetilde{\mu}\big(B_{\operatorname{H}}(s,r)\big)=0$ when $\mre(s)$ is large enough, say $\sigma>\sigma_0$, or when $|\mim(s)|$ is large enough, say $|t|>t_0$. This means that the integral \eqref{eq:lint} in our case is equal to
\[ \int_0^{\sigma_0} \int_{|t|\leq t_0} 	\left(\widetilde{\mu}\big(B_{\operatorname{H}}(s,r)\big)\right)^\frac{2}{2-p}\,\sigma^{\frac{2p-6}{2-p}}\,|s+1|^{\frac{4}{2-p}} \frac{dtd\sigma}{\pi}\asymp \int_0^{\sigma_0} \int_{|t|\leq t_0} \left(\widetilde{\mu}\big(B_{\operatorname{H}}(s,r)\big)\right)^\frac{2}{2-p}\,\sigma^{\frac{2p-6}{2-p}}\,dt\,d\sigma=:I.\]
	This means we only need to prove that $I<\infty$ for any fixed pair $(\sigma_0,t_0)$.  Because $ \widetilde{\mu}\big(B_{\operatorname{H}}(s,r)\big)$ is bounded, we may in fact assume that $\sigma_0$ is very small. Now, let us fix $s\in\CC_0$ with $\mre(s)\leq\sigma_0$ and let us consider $(\theta_1,\theta_2)\in [-\pi,\pi)^2$ such that $\widetilde{\Phi}(e^{i\theta_1},e^{i\theta_2})\in B_{\operatorname{H}}(s,r)$. Writing
	\[\mre\widetilde{\Phi}(e^{i\theta_1},e^{i\theta_2})=|c_{p_1}|\left(1-\cos(\theta_1)\right)+|c_{p_2}|\left(1-\cos(\theta_2)\right),\]
	it is clear that $\theta_1$ and $\theta_2$ are close to $0$, so that
	\[\theta_1^2+\theta_2^2\ll \mre \widetilde{\Phi}(e^{i\theta_1},e^{i\theta_2})\leq 2\sigma,\]
	and hence we conclude that $|\theta_1|,|\theta_2|\ll \sigma^{1/2}$. On the other hand, this implies that
	\[\left|\mim \widetilde{\Phi}(e^{i\theta_1},e^{i\theta_2})\right|=\left||c_{p_1}|\sin(\theta_1)+|c_{p_2}|\sin(\theta_2)\right|\ll\sigma^{1/2},\]
	which yields that $\widetilde{\mu}\big(B_{\operatorname{H}}(s,r)\big)=0$ provided $|t|\gg\sigma^{1/2}$. Otherwise, for a fixed value of $\theta_2$, we note that $\theta_1$ belongs to some interval with length dominated by $C\sigma$. Therefore, by Fubini's theorem, $\widetilde{\mu}\big(B_{\operatorname{H}}(s,r)\big)\ll \sigma^{{3/2}}$ where the involved constant does not depend on $t$. In total, this means that we require
	\[I \ll \int_0^{\sigma_0} \int_{|t|\ll\sigma^{1/2}} \sigma^{\frac{2p-3}{2-p}}\,dt\,d\sigma \asymp \int_0^{\sigma_0} \sigma^{\frac{2p-3}{2-p}+\frac 12}\,d\sigma < \infty.\]
This last integral is convergent for $p\geq 1$.
\end{proof}

\begin{rem}
	It is possible to generate more examples from the results in \cite{BB15} or from the results of Section \ref{sec:poly} in combination with Theorem~\ref{thm:bergmanemb}. If $2k< p<2k+2$, we can choose any Dirichlet polynomial with $\kappa\geq p/2k$, where $\kappa$ as defined in \cite[Lem.~10]{BB15}. However, this also illustrates the disadvantage of this interpolation method, since the natural condition is $\kappa\geq1$, which corresponds to the case $d=1$ in \eqref{eq:linear}.
\end{rem}

We end this section by emphasizing that results with $d=1$, or results for Dirichlet polynomials $\varphi\in\mathcal{G}$ with unrestricted range and $\dim(\varphi)=1$, cannot be obtained from the type of Carleson measure arguments employed in this section and the local embedding \eqref{eq:localemb} seems to be completely unavoidable in this setting.


\section{The multiplicative Hilbert matrix} \label{sec:hilbert}

It was asked in \cite[Sec.~6]{BPSSV14} whether the multiplicative Hilbert matrix introduced in the same paper has a bounded symbol on the polytorus $\mathbb{T}^\infty$, or, equivalently, whether the functional
\begin{equation} \label{eq:hilbertfunc}
	L(f) = \int_{1/2}^\infty (f(s)-a_1)\,ds
\end{equation}
is bounded on $\mathcal{H}^1$. It follows by standard Carleson measure techniques that if the embedding \eqref{eq:localemb} holds for $\mathcal{H}^p$, then the functional \eqref{eq:hilbertfunc} acts boundedly on $\mathcal{H}^p$. It is therefore only known that the functional is bounded on $\mathcal{H}^p$ when $2\leq p < \infty$.

Returning to the composition operator with symbol $\varphi(s)=3/2-2^{-s}$, we write out explicitly the associated Carleson measure, finding that boundedness of $\mathcal{C}_\varphi$ on $\mathcal{H}^p$ is equivalent to the inequality
\[\int_{3/2+\mathbb{T}}|P(z)|^p\,dm(z) \ll \|P\|_{\mathcal{H}^p}^p,\]
for every Dirichlet polynomial $P$.  If we apply the characterization of Carleson measures for $H^p(\mathbb{T})$, we find furthermore that
\[\int_{1/2}^1 |P(z)|^p\,dz \ll \int_{3/2+\mathbb{T}}|P(z)|^p\,dm(z).\]
From this and the results in \cite[Sec.~6]{BPSSV14} we observe that if $\mathcal{C}_\varphi$ acts boundedly on $\mathcal{H}^1$, then the multiplicative Hilbert matrix considered in \cite{BPSSV14} has a bounded symbol on the polytorus $\mathbb{T}^\infty$. In \cite{BPSSV14} it is only shown that the embedding \eqref{eq:localemb} implies that the multiplicative Hilbert matrix has a bounded symbol, so this observation is in some sense an improvement.

Using Theorem~\ref{thm:bergmanemb}, we can prove boundedness of $L$ on $\mathcal H^p$ for $p\in(1,\infty)$.
\begin{thm}\label{thm:hilbertmatrix}
The functional $L$ defined by \eqref{eq:hilbertfunc} is bounded on $\mathcal H^p$ for any $p>1$.
\end{thm}
\begin{proof}
We may restrict ourselves to $p\in (1,2)$. By Theorem~\ref{thm:bergmanemb}, it is sufficient to verify that the functional of integration from $1/2$ to 1 is bounded on $D_{2/p-1,\operatorname i}(\CC_{1/2})$ or, equivalently, that the functional of integration from 0 to 1 is bounded on $D_{2/p-1}(\DD)$. By duality, this is true since
\[f_\alpha(z) = \sum_{k=0}^\infty (k+1)^{\alpha-1} z^k\]
is in $D_\alpha(\mathbb{D})$ if and only if $\alpha<1$, so that $p>1$ is sufficient for $L$ to act boundedly on $\mathcal{H}^p$. 
\end{proof}
This theorem has an interesting corollary. Write $L(f)=\langle f,g \rangle_{\mathcal{H}^2}$, where
\[g(s) = \sum_{n=2}^\infty \frac{1}{\sqrt{n}\log{n}}n^{-s}.\]
We first have the following computation.
\begin{lem} \label{lem:hilbertsymb}
	$g \in \mathcal{H}^p$ if and only if $p<4$.
\end{lem}
\begin{proof}
	From the estimate $\sum_{n\leq x} [d(n)]^\alpha \asymp x(\log{x})^{2^\alpha-1}$ (see \cite{Wilson23}) and a standard computation with Abel summation, we find that
	\[\sum_{n=2}^\infty \frac{[d(n)]^\alpha}{n(\log{n})^\beta} < \infty\]
	if and only if $2^\alpha<\beta$. Assume first that $2<p<4$. Using \cite[Thm.~3]{Seip13} we have that
	\[\|g\|_{\mathcal{H}^p} \leq \left(\sum_{n=2}^\infty \frac{[d(n)]^{2-\frac{4}{p}}}{n(\log{n})^2}\right)^\frac{1}{2} < \infty,\]
	since $\alpha=2-4/p<1$ when $2<p<4$. For $p=4$, we compute
	\[\|g\|_{\mathcal{H}^4}^4 = \left\|g^2\right\|_{\mathcal{H}^2}^2 =\sum_{n=2}^\infty \frac{1}{n}\left(\sum_{\substack{d|n \\ 1<d<n}} \left(\log(d)\log(n/d)\right)^{-1}\right)^2 \gg \sum_{n=2}^\infty \frac{[d(n)]^2}{n(\log{n})^4} = \infty,\]
	so we are done.
\end{proof}
Theorem~\ref{thm:hilbertmatrix} and Lemma~\ref{lem:hilbertsymb} yield an explicit and natural example of the observation that $\mathcal{H}^q \subsetneq (\mathcal{H}^p)^\ast$ for H\"older conjugates $1<p,q<\infty$, as discussed in \cite[Sec.~3]{SS09}.
\begin{cor}
Let $1<p\leq4/3$ and set $1/p+1/q=1$. The Dirichlet series $g$ is in $(\mathcal H^{p})^* \backslash \mathcal{H}^q$.
\end{cor}

\bibliographystyle{amsplain} 
\bibliography{bergcomp} 
\end{document}